\documentclass[10pt,twoside,letterpaper]{article}

\usepackage[letterpaper,left=1.5in,right=1.5in,top=1.5in,bottom=1.5in]{geometry}
\title{The $\K$-theory of perfectoid rings}
\author{Benjamin Antieau, Akhil Mathew, and Matthew Morrow}
\date{\today}

\renewcommand{\phi}{\varphi}
\newcommand{\modfg}{\mathrm{Mod}^{\mathrm{fp}}}
\usepackage[pdfstartview=FitH,
            pdfauthor={Antieau},
            pdftitle={The K-theory of perfectoid rings},
            colorlinks,
            linkcolor=reference,
            citecolor=citation,
            urlcolor=e-mail,
            backref]{hyperref}

\usepackage{amsmath}
\usepackage{amscd}
\usepackage{amsbsy}
\usepackage{amssymb}
\usepackage{verbatim}
\usepackage{eufrak}
\usepackage{eucal}
\usepackage{microtype}
\usepackage{hyperref}
\usepackage{mathrsfs}
\usepackage{amsthm}
\usepackage{stmaryrd}
\usepackage{framed}
\usepackage{xy}
\input xy
\xyoption{all}
\usepackage{tikz}
\usetikzlibrary{matrix,arrows}
\usepackage{dsfont}
\usepackage[T1]{fontenc}

\usepackage{fancyhdr}

\usepackage{cleveref}

\usepackage{todonotes}

\usepackage{color}
\definecolor{todo}{rgb}{1,0,0}
\definecolor{conditional}{rgb}{0,1,0}
\definecolor{e-mail}{rgb}{0,.40,.80}
\definecolor{reference}{rgb}{.20,.60,.22}
\definecolor{mrnumber}{rgb}{.80,.40,0}
\definecolor{citation}{rgb}{0,.40,.80}

\let\oldmarginpar\marginpar
\renewcommand\marginpar[1]{\-\oldmarginpar[\raggedleft\footnotesize #1]%
{\raggedright\footnotesize #1}}

\newcommand{\Dscr}{\mathcal{D}}

\newcommand{\A}{\mathrm{A}}

\newcommand{\K}{\mathrm{K}}

\newcommand{\W}{\mathrm{W}}

\renewcommand{\mathbb}{\mathds}

\renewcommand{\AA}{\mathds{A}}

\newcommand{\FF}{\mathds{F}}

\newcommand{\ZZ}{\mathds{Z}}

\newcommand{\gr}{\mathrm{gr}}

\newcommand{\op}{\mathrm{op}}

\newcommand{\perf}{\mathrm{Perf}}

\newcommand{\fib}{\mathrm{fib}}

\newcommand{\spec}{\mathrm{Spec}}

\newcommand{\fun}{\mathrm{Fun}}
\newcommand{\catst}{\mathrm{Cat}_\infty^{\mathrm{perf}}}
\newcommand{\sF}{\mathcal{F}}

\newcommand{\G}{\mathrm{G}}
\newcommand{\xto}{\xrightarrow}

\renewcommand{\inf}{\mathrm{inf}}

\newcommand{\heart}{\heartsuit}

\renewcommand{\geq}{\geqslant}
\renewcommand{\leq}{\leqslant}

\newcommand{\KH}{\mathrm{KH}}

\newcommand{\Map}{\mathrm{Map}}

\newcommand{\Mod}{\mathrm{Mod}}

\newcommand{\Perf}{\mathrm{Perf}}

\DeclareMathOperator*{\holim}{holim}
\DeclareMathOperator*{\colim}{colim}

\newcommand{\Sch}{\mathrm{Sch}}
\DeclareMathOperator{\Spec}{Spec}

\newcommand{\we}{\simeq}
\newcommand{\iso}{\cong}

\theoremstyle{plain}
\newtheorem{theorem}{Theorem}[section]
\newtheorem*{theorem*}{Theorem}
\newtheorem{lemma}[theorem]{Lemma}

\newtheorem{proposition}[theorem]{Proposition}

\newtheorem{corollary}[theorem]{Corollary}
\newtheorem*{corollary*}{Corollary}

\theoremstyle{plain}

\theoremstyle{definition}

\newtheoremstyle{named}{}{}{\itshape}{}{\bfseries}{.}{.5em}{#1 \thmnote{#3}}
\theoremstyle{named}

\theoremstyle{definition}
\newtheorem{definition}[theorem]{Definition}

\newtheorem*{example*}{Example}

\newtheorem{question}[theorem]{Question}
\newtheorem*{question*}{Question}

\newtheorem{remark}[theorem]{Remark}

\newcommand{\bb}[1]{\mathbb{#1}}
\newcommand{\sub}[1]{{\mbox{\rm \scriptsize #1}}}
\newcommand{\quis}{\stackrel{\sim}{\to}}

\newcounter{MM}

\newcounter{BA}

\newcounter{AM}

\newcommand{\roi}{\mathcal O}

\begin{document}

\maketitle

\begin{abstract}
    \noindent
    We establish various properties of the $p$-adic algebraic $\K$-theory of smooth
    algebras over perfectoid rings living over perfectoid valuation rings. In
    particular, the $p$-adic $\K$-theory of such rings
    is homotopy invariant, and coincides with the $p$-adic $\K$-theory of the $p$-adic
    generic fibre in high degrees. In the case of smooth algebras over perfectoid valuation rings of
    mixed characteristic the latter isomorphism holds in all degrees and
    generalises a result of Nizio\l.
\end{abstract}

\section{Introduction}
In this note, we record some results concerning the $p$-adic $\K$-theory of
certain $p$-adic rings.  Our starting point is the following result.

\begin{theorem}[Quillen, Hiller \cite{Hiller81}, Kratzer \cite{Kra80}]\label{QHKthm}
If $A$ is a perfect $\mathbb{F}_p$-algebra, then $\K_i(A)$ is a $\mathbb{Z}[1/p]$-module for all $i > 0$. 
\end{theorem} 

\Cref{QHKthm} is proved using the action of the Adams operations on
$\K$-theory. In particular, one shows that $\psi^p$ is given by the Frobenius,
and therefore is an isomorphism. The mixed characteristic analog of a perfect
$\mathbb{F}_p$-algebra is a perfectoid ring, and many foundational results for
perfect $\mathbb{F}_p$-algebras can be generalized to perfectoid rings. We
begin by giving the following generalization of \Cref{QHKthm} to mixed
characteristic. In the statement, $\K(-;\FF_p)$ denotes the cofiber $\K(-)/p$ of
multiplication by $p$ on non-connective $\K$-theory $\K(-)$.

\begin{theorem}[\Cref{Invertingp}]
    If $\mathcal{O}$ is a perfectoid valuation ring and $A$ is a perfectoid
    $\mathcal{O}$-algebra,\footnote{For us, a ``perfectoid valuation ring'' is a
    valuation ring $\mathcal{O}$ which is simultaneously a perfectoid ring in the sense
    of~\cite[Def.~3.5]{BMS1}: $\mathcal{O}$ is $p$-adically complete, there is
    an element $\varpi$ in $\mathcal{O}$ such that $\varpi^p$ divides $p$, the Frobenius map on
    $\mathcal{O}/p$ is surjective, and the kernel of
    $\theta\colon\A_{\inf}(\mathcal{O})\rightarrow \mathcal{O}$
    is principal; such a valuation ring is $p$-adically complete and separated.
    Equivalently, it is either a perfect valuation ring of characteristic $p$,
    or its fraction field is a perfectoid field of characteristic $0$ having
    ring of integers $\mathcal{O}_{\sqrt{p\mathcal{O}}}$.
    A ``perfectoid $\mathcal{O}$-algebra'' is a perfectoid ring $A$
    equipped with the structure of a $\mathcal{O}$-algebra.}
then the map $\K(A; \mathbb{F}_p) \to \K( A[1/p]; \mathbb{F}_p)$ is
    $0$-truncated.\footnote{Recall that a map $f\colon M\rightarrow N$ of
    spectra is $d$-truncated if the fiber $\fib(f)$ satisfies $\pi_i\fib(f)=0$
    for $i>d$. Equivalently, $\pi_if\colon\pi_iX\rightarrow\pi_iY$ is an
    isomorphism for $i>d+1$ and an injection for $i=d+1$.}
\label{perfectoidk}
\end{theorem} 

 \Cref{perfectoidk} holds more generally for any ring $A$ whose derived
 $p$-completion satisfies the hypotheses of the theorem, as the conclusion of
 the theorem is insensitive to replacing $A$ by its derived $p$-completion.  In
 the case where $A$ is the absolute integral closure of a complete discrete
 valuation ring of mixed characteristic, \Cref{perfectoidk} is proved by
 different methods in 
\cite{niziol-crystalline} and \cite{hesselholt-ocp}. In fact, the result in
\cite{niziol-crystalline} works more generally for smooth algebras, and plays a
crucial role in the approach to the crystalline conjecture of $p$-adic Hodge
theory in \emph{op.~cit}.

\begin{theorem}[{Nizio\l~\cite[Lem.~3.1]{niziol-crystalline}}] 
\label{Nizthm}
Let $\mathcal O_K$ be a complete discrete valuation ring with fraction field $K$; let $ \overline{K}$ be an
algebraic closure of $K$, and $\mathcal O_{\overline K}$ the integral closure of $\mathcal O_K$ in $\overline K$. 
    For any smooth $\mathcal O_{\overline{K}}$-algebra $R$, the natural map 
$\K(R) \to \K( R \otimes_{\mathcal O_{\overline K}} \overline{K})$ becomes an equivalence
after profinite completion. 
\end{theorem} 

\Cref{Nizthm} is proved using localization sequences in $\K$-theory after
descending $R$ to the integral closure of $\roi_K$ in a finite extension of $K$; more
generally, one can replace $\roi_{\overline K}$ by any absolutely integrally closed 
valuation ring (\Cref{generalNizthmAIC}). 
Combining this argument with the  tilting correspondence 
\cite{Sch12}
to reduce to
characteristic $p$ and the inseparable local uniformization of Temkin
\cite{Temkin13, Tem17}, 
we give the following generalization of  (the $p$-adic
case of) \Cref{Nizthm} as well as of \Cref{perfectoidk}. We remark that the theorem holds more generally for algebras which are $p$-completely smooth in a suitable sense (see \Cref{corol_t-smooth} and~\Cref{rem_p_smooth}).

\begin{theorem}[\Cref{smoothoverperfdval} and~\Cref{Invertingp}]
\label{generalNizthm}
    Let $\mathcal{O}$ be a perfectoid valuation ring.
\begin{enumerate}
    \item If $p\neq 0$ in $\mathcal{O}$, then for any smooth
        $\mathcal{O}$-algebra $R$, the natural map $\K(R; \mathbb{F}_p) \to
        \K(R[1/p]; \mathbb{F}_p)$ is an equivalence.
    \item For any perfectoid $\mathcal{O}$-algebra $A$ and smooth $A$-algebra
        $R$ of relative dimension $\le d$, the map $\K(R; \mathbb{F}_p) \to
        \K(R[1/p]; \mathbb{F}_p)$ is $d$-truncated.
\end{enumerate}
\end{theorem}

For any ring $R$, the natural map $\K(R) \to \K(R[1/p])$
becomes an equivalence after $K(1)$-localization \cite{BCM20, LMMT20}, but in
general the conclusion that they agree in sufficiently high degrees requires
further assumptions. 
The proof deduces part 2 from part 1 using cdh-descent and separate arguments in
the case of valuation rings. 
These arguments also lead to the following result comparing algebraic and
homotopy $\K$-theory for perfectoid rings; since the comparison between $\K$ and
$\KH$ of a noetherian ring is known to be related to singularities, the following result
gives a sense in which perfectoid rings behave like regular ones.

\begin{theorem}[\Cref{KisKHperfect}, \Cref{KisKHpadicWittperfect}, and \Cref{KisKHperfectoid}]
Let $R$ be a smooth algebra over either
\begin{enumerate}
\item a perfect $\bb F_p$-algebra; or
\item $W(A)$ where $A$ is a perfect $\bb F_p$-algebra; or
\item a perfectoid ring which is an algebra over some perfectoid valuation ring.
\end{enumerate}
Then the map $\K(R; \mathbb{F}_p) \to \KH(R; \mathbb{F}_p)$ is an equivalence (in case 1, even $\K(R)\rightarrow\KH(R)$ is an equivalence).  
\end{theorem}

\subsection*{Notation}
We let $\catst$ denote the $\infty$-category of small,
stable idempotent-complete $\infty$-categories. 
We denote by $\K$ the nonconnective $\K$-theory functor, defined on $\catst$
as in \cite{bgt1} and taking values in spectra. Similarly, we denote by $\KH$ the homotopy $\K$-theory of
\cite{WeibelKH}, defined more generally on $H\mathbb{Z}$-linear
$\infty$-categories \cite{Tab15}.

Throughout, we let $\K(-;\ZZ_p)$ and $\KH(-;\ZZ_p)$ denote the $p$-completions of
$\K$-theory and homotopy invariant $\K$-theory.
Similarly, we denote by $\K(-; \mathbb{F}_p)$ and $\KH(-; \mathbb{F}_p)$ their
mod $p$ reductions. 

All rings in this paper will be commutative.
Given a ring $R$, we let $\mathcal{D}(R)$ denote its derived $\infty$-category. 
Given a ring $R$ (or more generally an $\mathbb{E}_\infty$-ring) and an ideal $I \subset R$, we let
$\mathrm{Perf}(R \text{ on } I)$ denote the $\infty$-category of perfect
$R$-module spectra $M$ which are $I$-power torsion: in other words, for any $x
\in I$, $M[x^{-1} ] = 0$. 
We will only use this definition when $I$ is the radical of a finitely generated
ideal $J$, in which case $\mathrm{Perf}(R \text{ on } I)$ is the kernel of
$\mathrm{Perf}( \spec(R)) \to \mathrm{Perf}( \spec(R) \setminus V(J))$. 
Given a localizing invariant $E$, we write $E(R \text{ on } I) = E(
\mathrm{Perf}(R \text{ on } I))$. 
Given a map of pairs $(R, I) \to (S, J)$ such that $\mathrm{rad}(IS) = J$, 
base-change induces a functor $\mathrm{Perf}(R \text{ on } I) \to
\mathrm{Perf}(S \text{ on } J )$ and a consequent map in any localizing
invariant. 

We adopt the convention in this paper that localizing invariants commute with filtered colimits.
\subsection*{Acknowledgments}

We thank Dustin Clausen, Lars Hesselholt,  and Wies{\l}awa Nizio\l\ for helpful discussions. 
This material is based upon work supported by the National Science Foundation
under Grant No. DMS-1440140 while the authors were in residence at the Mathematical
Sciences Research Institute in Berkeley, California, during the Spring 2019 semester.
The first author was supported by NSF Grants DMS-2102010 and DMS-2120005 and a Simons Fellowship.
This work was done while the second author was a Clay Research Fellow. This
project has received funding from the European Research Council (ERC) under the
European Union's Horizon 2020 research and innovation programme (grant
agreement No. 101001474).

\section{Localization sequences}

In this section, we review some basic properties of coherent rings 
and their
$\K$-theory, in particular proving the localization
results \Cref{Quillendevcoherent} and \Cref{locregdivisor}. All of these are direct analogs of standard properties
\cite{Qui72} of the $\K$-theory
and $\G$-theory of noetherian schemes; we will need to apply them to valuation
rings. In Appendix A, we indicate how to prove these results and some
generalizations using d\'evissage results about the $\K$-theory of stable
$\infty$-categories (which will not be used in the rest of the paper); in this section we only use classical 
d\'evissage theorems. 

We will say that a coherent ring $R$ is \emph{weakly regular} if $R$ has
finite flat
(or weak) dimension. Equivalently, by \cite[Cor.~2.5.6]{Glaz}, the projective dimensions of
finitely presented $R$-modules are uniformly bounded (necessarily by the flat dimension). 
A ring $R$ is said to be \emph{stably coherent} if every finitely presented 
$R$-algebra is coherent. It is sufficient to check coherence of finitely generated polynomial algebras over $R$. The class of stably coherent rings is closed
under localizations, quotients by finitely generated ideals, and finitely presented extensions \cite[Thms.~2.4.1 \& 2.4.2]{Glaz}. We will primarily be interested in weakly regular stably coherent rings; this includes all regular\footnote{Throughout the article, we adopt the convention that regular rings are assumed to be Noetherian.}
rings of finite Krull dimension, but also valuation rings by the following results.

\begin{proposition} 
\label{valstcoh}
Any valuation ring is stably coherent and of flat dimension $\leq 1$. 
\end{proposition} 
\begin{proof} 
The stable coherence is \cite[Th.~7.3.3]{Glaz}. Every torsion-free module over a valuation ring is
flat \cite[Tag 0549]{stacks-project}, whence the second claim. 
\end{proof} 

\begin{proposition} 
\label{fflatpassesup}
    Let $A$ be a stably coherent ring with flat dimension $d_0$, and let $R$ be a
    smooth algebra of relative dimension\footnote{We say that a smooth $A$-algebra $R$ has {\em
    relative dimension $\le d$} if all fibers
    $R\otimes_{A}\kappa(\frak p)$, where $\kappa(\frak p)$ runs over the
    residue fields of $A$, have Krull dimension $\le d$.} $\le d$ over $A$.
    Then $R$ has flat dimension $\leq d + d_0$; in particular, $R$ is weakly
    regular.
\end{proposition} 
\begin{proof} 
Let $M$ be a finitely presented $R$-module. 
It suffices to show that if $M$ is flat as an $A$-module, then $M$ has
projective dimension $\leq d$ as an $R$-module. 
By \cite[Cor.~2.5.10]{Glaz}, it suffices to show that for every maximal ideal 
$\mathfrak{m}$ of $R$, one has $\mathrm{Tor}_R^{d + 1}(M, R/\mathfrak{m}) = 0$. 
Now $R/\mathfrak{m}$ pulls back to a  prime ideal $\mathfrak{p} \subset A$
with residue field $\kappa(\mathfrak{p})$. 
Since $M$ is flat over $A$, we have
$\mathrm{Tor}_R^{d+1}(M, R/\mathfrak{m}) = 
\mathrm{Tor}^{d+1}_{R \otimes_{A} \kappa(\mathfrak{p})}( M \otimes_{A}
\kappa(\mathfrak{p}), R/\mathfrak{m})$. However, this vanishes 
since $R \otimes_{A} \kappa( \mathfrak{p})$ is a smooth algebra over
the field
$\kappa(\mathfrak{p})$ of dimension $\le d$ and consequently it has global dimension $\le d$. 
\end{proof} 

\begin{corollary}\label{smooth_over_val_are_nice}
Any smooth algebra over a valuation ring is weakly regular stably coherent.
\end{corollary}

The $\K$-theory of weakly regular stably coherent rings behaves in a similar
way to that of regular Noetherian rings, as exemplified by the following
result. Given a coherent ring $R$, we define the \emph{$\G$-theory} $\G(R)$ to
be the connective $\K$-theory of the abelian category of finitely presented
$R$-modules. The first two parts of
the next result appear as \cite[Ex.~1.4]{WeibelKH}; compare also
\cite[Th.~3.33]{AGH} and \cite[Th.~3.3]{KM21} for treatments.

\begin{proposition} 
\label{stablycohfflat}
If $R$ is a weakly regular stably coherent ring, then
\begin{enumerate}
\item $\K_{-i}(R) = 0$ for $i > 0$; that is, the canonical map
$\K_{\ge0}(R)\to\K(R)$ from connective $\K$-theory to $\K$-theory is an equivalence;
\item the natural map $\K(R) \to \KH(R)$ is an equivalence;
\item the natural map $\K_{\ge0}(R)\to\G(R)$ is an equivalence.
\end{enumerate}
\end{proposition} 
\begin{proof}
We have already explained that the first two parts may be found in  \cite[Ex.~1.4]{WeibelKH}, where we implicitly use \Cref{fflatpassesup}. They are also special cases of \Cref{thmofheart} and 
\Cref{KisKHwhenboundedt}. 

The third part follows from Quillen's d\'evissage theorem, as the hypotheses imply that any object in the abelian category of finitely presented $R$-modules admits a finite length resolution by finite projective modules. Alternatively it is a special case of \Cref{thmofheart}.
\end{proof}

We may now present the localization sequences which will be required later.

\begin{proposition} 
\label{Quillendevcoherent}
Let $R$ be a weakly regular stably coherent ring and let $ I \subset R$ a
finitely generated ideal. Then there is a natural fiber sequence 
$\G(R/I) \to \K(R) \to \K( \spec(R)\setminus V(I))$. 
\end{proposition} 
\begin{proof} 
    We write $\G(\spec(R)\setminus V(I))$ for the connective $\K$-theory of the
    abelian category of finitely presented quasi-coherent sheaves on $\spec(R)
    \setminus V(I)$. This abelian category is the Serre quotient of the abelian
    category of finitely presented $R$-modules by the subcategory of those
    objects which are $I$-power torsion; see in particular \cite[Tag
    01PD]{stacks-project} for the result that finitely presented quasi-coherent
    sheaves can be extended so that $\K_0(R)\rightarrow\K_0(\Spec(R)\setminus
    V(I))$ is surjective.

    The classical localization and d\'evissage theorems
    \cite[Thms.~4 and 5]{Qui72}
    therefore provide a fiber sequence $\G(R/I)\to\G(R)\to\G(\spec R \setminus
    V(I))$. Finally,  \Cref{stablycohfflat} implies that $\G(R)\simeq\K(R)$ and
    that $\G(\spec(R) \setminus V(I))\simeq\K(\spec(R)\setminus V(I))$ (in the
    latter case use induction on the size of a finite affine open cover of
    $\spec(R)\setminus V(I)$, again using \cite[Tag 01PD]{stacks-project} to
    eliminate any possible problem with failure of surjectivity on $\K_0$).
\end{proof} 

\begin{proposition} 
    If $R$ is a ring and $t\in R$ is a nonzerodivisor such that $R/t$ is weakly
    regular stably coherent, then there are natural fiber sequences
    $\K(R/t)\rightarrow\K(R)\rightarrow\K(R[1/t])$ and
    $\KH(R/t)\rightarrow\KH(R)\rightarrow\KH(R[1/t])$.
	 Consequently, we have a pullback square
     \begin{equation} \label{KtoKH} \begin{gathered}\xymatrix{
	 \K(R) \ar[d]  \ar[r] &  \K(R[1/t]) \ar[d]  \\
	 \KH(R) \ar[r] &  \KH(R[1/t]).
     }\end{gathered}\end{equation}
\label{locregdivisor}
\end{proposition} 
\begin{proof}
Note first that $R/t^n$ is stably coherent for each $n \geq 1$ by
\cite[Lem.~3.26]{BMS1}. 
By the localization theorem of \cite{Grayson}, 
the connective cover of the fiber of $\K(R) \to \K(R[1/t])$ (i.e., the
    connective cover of $\K(R\textrm{ on }tR)$) is the $\K$-theory of the exact category
$\mathcal{E}$ of 
finitely presented $R$-modules $M$ such that $M[1/t] = 0$ and such that $M$ has
$\mathrm{Tor}$-dimension $\leq 1$.

Consider the category $ \mathcal{A}$ of all finitely presented
$R$-modules which are $t$-power torsion. 
In other words, $\mathcal{A}$ is the union of the categories of finitely
presented $R/t^n$-modules over all $n \geq 0$. Our coherence hypotheses thus show that
$\mathcal{A}$ is an abelian category, and $\mathcal{A}$ contains $\mathcal{E}$
as an exact subcategory. 

We observe that every object in $\mathcal{A}$ has finite
$\mathrm{Tor}$-dimension as an $R$-module. 
Indeed, suppose $M \in \mathcal{A}$ is a finitely presented $R$-module with $M[1/t]
= 0$. 
We may assume $tM =0 $. Then our weak regularity hypothesis implies that $M$ has
finite $\mathrm{Tor}$-dimension as an $R/t$-module, and hence as an $R$-module. 
Using this, we can show that  every object in $\mathcal{A}$ admits a finite resolution by objects
in $\mathcal{E}$. If $M \in \mathcal{A}$ has $\mathrm{Tor}$-dimension $ \geq 2$,
then 
we can choose a surjection $(R/t^i)^n \twoheadrightarrow M$ (for appropriate $i,
n \gg 0$); the kernel $K$ will belong to $\mathcal{A}$ and have
$\mathrm{Tor}$-dimension at least one less, whence the claim by induction. 

Thus we can apply the d\'evissage theorem in the 
form of 
\cite[Sec.~4]{Qui72} to see that $\K(\mathcal{E}) \xrightarrow{\sim}
\K(\mathcal{A})$. By d\'evissage again, we have $\K(\mathcal{A}) = \G(R/t)$,
which is $\K(R/t)$ by \Cref{stablycohfflat} because $R/t$ is stably coherent
    and weakly regular. In conclusion, we have shown that the canonical map
    $\K(R/t)\to\K(R\textrm{ on }tR)$ identifies the left side with the
    connective cover of the right side.

To obtain the result in nonconnective degrees, and so complete the proof, we
    claim that the canonical map $\K_i(R/t)\to \K_i(R\textrm{ on }tR)$ is an
    isomorphism for all $i\le0$. The case $i=0$ has already been proved, so we
    proceed inductively by Bass delooping via the fundamental theorem of
    $\K$-theory \cite[Thm.~6.1]{TT90}; this
gives exact sequences
\[
\K_i(A[u]) \oplus \K_i(A[u^{-1}]) \to \K_i(A[u^{\pm 1}]) \to
\K_{i-1}(A) \to 0\]
for any ring $A$, and more generally
\[
\K_i(A[u]\textrm{ on }IA[u]) \oplus \K_i(A[u^{-1}]\textrm{ on }IA[u^{-1}]) \to \K_i(A[u^{\pm 1}]\textrm{ on }IA[u^{\pm1}]) \to
\K_{i-1}(A\textrm{ on }I) \to 0
\]
for any finitely generated ideal $I\subseteq A$. Assuming that the claim has
    been proved in any fixed degree $i\le 0$ (for all pairs $R,t$ as in the
    statement of the proposition), one immediately obtains the claim in degree
    $i-1$ by comparing the Bass exact sequences for $\K(R/t)$ and $\K(R\textrm{
        on }tR)$, noting that the inductive hypothesis for the pair $R[u],t$
    implies that $\K_i(R[u]/t)\to\K_i(R[u]\textrm{ on }tR[u])$ is an isomorphism,
    and similarly for $R[u^{-1}]$ and $R[u^{\pm1}]$.
\end{proof} 

\section{Smooth algebras over valuation rings}
In this section we study the $\K$-theory of smooth algebras over valuation rings, and in particular prove \Cref{generalNizthm}(1).

The proof of the following lemma is close to that of \cite[Lem.~3.1]{niziol-crystalline}:

\begin{lemma}
\label{divmap}
Let $(A_1, \mathfrak{m}_1) \to (A_2, \mathfrak{m}_2)$ be a finite flat map of regular local
rings; let $R_1$ be a smooth $A_1$-algebra and set $R_2 = R_1 \otimes_{A_1} A_2$. 
Then the natural map (given by extension of scalars)
$\K( R_1 \text{ on } \mathfrak{m}_1R_1) \to 
\K(R_2 \text{ on } \mathfrak{m}_2R_2)$
is divisible by the integer $d = \mathrm{len}_{A_2}(A_2/\mathfrak{m}_1 A_2)$. 
\end{lemma} 
\begin{proof} 
Consider the functor of stable $\infty$-categories
\[ F\colon  \perf(R_1/\mathfrak{m}_1R_1) \to \perf( R_1 \text{ on }\mathfrak{m}_1R_1)
    \xrightarrow{(-)\otimes_{A_1} A_2} 
\perf( R_2 \text{ on } \mathfrak{m}_2R_2).  \]
By d\'evissage, the first map induces  an equivalence on $\K$-theory. 
The composite functor $F\colon \perf(R_1/\mathfrak{m}_1) \to \perf( R_2 \text{ on }
\mathfrak{m}_2R_2)$ is equivalently given by 
the tensor product functor $(-)\otimes_{A_1/\mathfrak{m}_1} A_2/\mathfrak{m}_1 A_2$. 
We need to show that $F$ induces a map on $\K$-theory which is
divisible by $d$; it suffices to show that in 
$\K_0( \fun( \perf(R_1/\mathfrak{m}_1), \perf(R_2 \text{ on } \mathfrak{m}_2R_2)))$,
the class $[F]$ is divisible by~$d$. 

Any finite length $A_2/\mathfrak{m}_1A_2$-module $M$ induces a functor
$(-)\otimes_{A_1/\mathfrak{m}_1} M\colon \perf(R_1/\mathfrak{m}_1) \to \perf( R_2
\text{ on } \mathfrak{m}_2R_2)$, from which we obtain a class in $\K_0( \fun( \perf(R_1/\mathfrak{m}_1), \perf(R_2 \text{ on } \mathfrak{m}_2R_2)))$; moreover, this process takes short exact sequences of modules to sums in the $\K_0$-group. Since $A_2/\mathfrak{m}_1 A_2$ has a finite filtration with associated graded given by $d$ copies of $A_2/\mathfrak{m}_2$, 
it follows that $[F]$ is equal to $d$ times the class of the functor
$(-)\otimes_{A_1/\mathfrak{m}_1} A_2/\mathfrak{m}_2$. 
\end{proof}

An ind-regular local ring is a local ring which is a filtered colimit of regular 
rings (without loss of generality, one can take a filtered colimit of regular
local rings under local homomorphisms).

\begin{proposition} 
\label{modpzero}
Let $A$ be an ind-regular local ring with maximal
ideal $\mathfrak{m}_A$, and let $d\ge1$; assume that $\mathfrak{m}_A$ is the radical of a finitely
generated ideal in $A$, and that every element of $\mathfrak{m}_A$
is a $d^\sub{th}$ power. Then $\K( R \text{ on } \mathfrak{m}_AR)/d =0$ for
    every smooth $A$-algebra $R$, i.e., $\K(R)/d\quis \K(\Spec(R)\setminus
    V(\mathfrak m_AR))/d$.
\end{proposition} 
\begin{proof} 
Let $x$ be a class in $\pi_n( \K( R \text{ on } \mathfrak{m}_AR)/d)$ for some integer $n$; we need to show that $x$ vanishes.
By assumption $A$ is a filtered colimit of regular local rings, so there exists a regular local ring  $(A_0, \mathfrak{m}_0)$, a map of local rings
$(A_0, \mathfrak{m}_0) \to (A, \mathfrak{m}_A)$, a smooth $A_0$-algebra $R_0$ with $R_0 \otimes_{A_0} A \simeq R$, and a class $x_0 \in
\pi_n( K ( R_0 \text{ on } \mathfrak{m}_0R_0)/d)$ such that $x_0$ is carried to $x$ under the
natural map 
\begin{equation}  \label{natmapR0}  \K ( R_0 \text{ on } \mathfrak{m}_0R_0)/d \to
\K ( R \text{ on }
\mathfrak{m}_AR)/d .\end{equation}
As $\frak m_A$ is the radical of a finitely generated ideal, we may further assume that $\mathrm{rad}(\mathfrak{m}_0 A) = \mathfrak{m}_A$.

If $A_0$ is a field then we automatically have $x = 0$, so suppose $\dim(A_0) \geq 1$. 
Let $\alpha \in \mathfrak{m}_{0} \setminus \mathfrak{m}_0^2$, and define the
local ring $(A_1, \mathfrak{m}_1)$ via $A_1 = 
A_0[t]/(t^{d^2} - \alpha)$. Note that $A_1$ is a regular local ring 
(with $t$ part of a system of parameters), and that $A_0 \to A_1$ is finite flat
with $d^2 =\mathrm{len}(A_1/\mathfrak{m}_0 A_1)$. 
Our assumptions imply that the map $(A_0, \mathfrak{m}_0) \to (A, \mathfrak{m}_A)$
factors over the inclusion $(A_0, \mathfrak{m}_0) \to (A_1, \mathfrak{m}_1)$. 
Therefore, the map 
\eqref{natmapR0}
factors over the map 
$ \K ( R_0 \text{ on } \mathfrak{m}_0R_0)/d  \to 
\K( R_1 
 \text{ on } \mathfrak{m}_1R_1)/d $, where $R_1 = R_0 \otimes_{A_0} A_1 $. 
 But this map is zero by 
\Cref{divmap}. 
\end{proof}

Consequently we recover \cite[Lem.~3.1]{niziol-crystalline}, generalized to
arbitrary absolutely integrally closed valuation rings. 
\begin{corollary} 
\label{generalNizthmAIC}
Let $V$ be an absolutely integrally closed valuation ring, and let $R$ be a
smooth $V$-algebra. Then the natural map $\K(R) \to \K(R \otimes_V
\mathrm{Frac}(V))$ is an equivalence after profinite completion. 
\end{corollary} 
\begin{proof} 
Writing $V$ as a filtered colimit of finite rank absolutely integrally closed
valuation rings, we may assume that $V$ has finite rank. 
In particular, in this case the maximal ideal $\mathfrak{m}_V \subset V$ is the
radical of $(t)$, for any $t \in \mathfrak{m}_V$ which does not belong to a
smaller prime ideal. 
Using induction on 
the rank of the valuation (note by elementary properties of valuation rings that $V[1/t]$ is a valuation ring of rank one lower than $V$, unless $V$ is a field in which case we are done), we are therefore reduced to showing that 
$\K(R \text{ on } \mathfrak{m}_VR) = \mathrm{fib}( \K(R) \to \K( R[1/t]))$ vanishes after profinite completion. 
But $V$ is ind-regular by a result of Temkin \cite{Tem17} as observed by
Elmanto and Hoyois (see
\cite[Cor.~4.2.4]{AD21} for a discussion), so \Cref{modpzero} applies for all $d \geq 1$ to complete the proof.
\end{proof}

\begin{corollary} 
\label{caseperfectvalring}
Let $V$ be a perfect valuation ring of
characteristic $p$, and let $R$ be a smooth $V$-algebra. Then 
$\K(R; \mathbb{Z}_p)\to \K(R \otimes_V \mathrm{Frac}(V);
\mathbb{Z}_p)$ is an equivalence.
\end{corollary} 
\begin{proof} 
This is proved exactly as in \Cref{generalNizthmAIC}, where we use
\Cref{modpzero} with $d = p$. Here we use purely inseparable local
uniformization \cite{Temkin13} to see that $V$ is ind-smooth over $\bb F_p$.
\end{proof} 

\begin{remark}[Motivic refinements]\label{motivic}
Let $X$ be a qcqs scheme over $\mathbb{F}_p$, and suppose that $X$ is the filtered
limit of a diagram of smooth $\mathbb{F}_p$-schemes along affine transition
maps. Then, defining motivic cohomology $\mathbb{Z}(i)^\sub{mot}(X)$ as the
filtered colimit of the motivic cohomologies of the smooth $\bb F_p$-schemes, we
see from \cite{FS02} and \cite{Lev08} (again by taking the filtered colimit), that $\K(X)$
admits a ``motivic filtration'' whose graded pieces $\gr^i \K(X)$ are
$\mathbb{Z}(i)^\sub{mot}(X)[2i]$ for $i\ge0$. When $\K$-theory admits such a motivic filtration it is natural to ask whether our results can be upgraded to filtered equivalences. For example, for $R$ a smooth algebra over a perfect valuation ring $V$ of characteristic $p$, the motivic refinement of Corollary \ref{caseperfectvalring} states that $\mathbb{Z}(i)^\sub{mot}(R)/p \xrightarrow{\sim} \mathbb{Z}(i)^\sub{mot}(R
\otimes_V \mathrm{Frac}(V))/p$ for all $i\ge0$; in this remark we show that this is indeed true.

Repeating the proof of Corollary \ref{caseperfectvalring}, it is enough to
    establish the following motivic variant of Proposition \ref{modpzero} (in
    the characteristic $p$ context): let $A$ be an ind-smooth local $\mathbb{F}_p$-algebra, such that $A$ is perfect and its maximal ideal $\mathfrak{m}_A$ is the radical of a finitely generated ideal. Then, for any smooth (or ind-smooth) $A$-algebra $R$, the canonical maps $\mathbb{Z}(i)^\sub{mot}( R   )/p \to \mathbb{Z}(i)^\sub{mot}(
\spec (R) \setminus V(\mathfrak{m}_AR))/p$ are equivalences for all $i\ge0$.

To prove this, we can assume that $R$ is local and essentially smooth over
$A$. Then, by the Geisser--Levine theorem \cite{GL00}, the motivic filtration on $\K( R; \mathbb{F}_p)$ is just the Postnikov filtration. Since we already know that $\K( R; \mathbb{F}_p) \xrightarrow{\sim} \K(\spec(R) \setminus V(\mathfrak{m}_AR); \mathbb{F}_p)$ by \Cref{modpzero}, it remains to to show that the motivic filtration on 
$\K(\spec(R) \setminus V(\mathfrak{m}_AR); \mathbb{F}_p)$ is also the Postnikov
filtration, or in other words that 
$\mathbb{Z}(i)^\sub{mot}(
\spec (R) \setminus V(\mathfrak{m}_AR))/p$ is concentrated in cohomological degree
$i$. In one direction, we know (again by \cite{GL00}) that 
$\mathbb{Z}(i)^\sub{mot}( \spec(R) \setminus V(\mathfrak{m}_AR))/p$
is concentrated in cohomological
degrees $\geq i$; it remains to prove the bound in the other direction.
Now we can write 
the pair $(A, \mathfrak{m}_A)$ as a filtered colimit of 
of  essentially smooth, local $\mathbb{F}_p$-algebras $(A_0,
\mathfrak{m}_{A_0})$ with maps $(A_0, \mathfrak{m}_{A_0}) \to (A,
\mathfrak{m}_A)$ such that $\mathfrak{m}_{A_0}$ generates $\mathfrak{m}_A$ up to
radical; we similarly write $R$ as a filtered colimit of 
algebras $R_0$ which are essentially smooth and local over such $A_0$. 
    Then the Gysin sequence in motivic cohomology (see~\cite[Thm.~15.15]{MVW})
\[\mathbb{Z}(i-d)^\sub{mot}( R_0/\mathfrak{m}_{A_0} )[-d]/p \to 
\mathbb{Z}(i)^\sub{mot} ( R_0)/p \to \mathbb{Z}(i)^\sub{mot}( \spec(R_0) \setminus V(\mathfrak{m}_{A_0} R_0))/p\] (where $d = \mathrm{dim}(A_0)$) shows that 
$\mathbb{Z}(i)^\sub{mot}( \spec(R_0) \setminus V(
\mathfrak{m}_{A_0} R_0))/p$ is concentrated in cohomological degrees $\le i$. Passing to the limit yields the same bound for $\mathbb{Z}(i)^\sub{mot}( \spec(R) \setminus V(\mathfrak{m}_AR))/p$ and so completes the proof.
\end{remark}

Although we will not need it, we record a final corollary of Proposition
\ref{modpzero} which extends Corollary \ref{caseperfectvalring} to smooth
algebras over arbitrary ind-smooth perfect domains. To bridge the gap between
the punctured spectrum of Proposition \ref{modpzero} and the full field of
fractions we must first prove the next lemma. For a noetherian spectral space
$X$ of finite Krull dimension and $x \in X$,
we let $X_x$ denote the space of all generizations of $x$, i.e., the
intersection of all open subsets containing $x$; note that $X_x$ is itself a noetherian spectral space with the subspace
topology. We recall also if $X$ is irreducible, then constant sheaves on $X$
have no higher cohomology \cite[Tag 02UU]{stacks-project}, so constant sheaves
and presheaves of spectra are the same, and constant sheaves are pushed forward from the
generic point.

\begin{lemma} 
\label{spectrallemma} 
Let $X $ be an irreducible, noetherian spectral space of finite Krull dimension. 
Let $\mathcal{G}$ be a sheaf 
of spectra on $X$; we extend $\mathcal{G}$ by continuity to all pro-open subsets of $X$. 
Suppose that for each $x\in X$, we have 
$\mathcal{G}( X_x) \xrightarrow{\sim} \mathcal{G}( X_x \setminus \left\{x\right\})$. 
Then $\mathcal{G}$ is constant. That is, for every nonempty open subset $U
\subset X$, we have $\mathcal{G}(X) \xrightarrow{\sim} \mathcal{G}(U)$.  
\end{lemma} 
\begin{proof} 
Let $\eta$ be the generic point of $X$, and let $\mathcal{G}_{\eta}$ denote the 
stalk of $\mathcal{G}$ at the generic point. We claim that $\mathcal{G}$ is the
constant sheaf (or presheaf) with value $\mathcal{G}_\eta$. 
 Note also that equivalences of sheaves of spectra can be detected
on stalks by our assumptions and  \cite[Cor.~7.2.4.20]{HTT}.

To prove the claim, we induct on the Krull dimension of $X$. 
For each $x \in X$ with $x \neq \eta$, we need to see that the generization map
$\mathcal{G}_x = \mathcal{G}(X_x) \to \mathcal{G}_\eta$ is an equivalence. 
However, 
by induction on the dimension of $X$, 
we find that $\mathcal{G}$ defines the constant presheaf on the pro-open subset $X_x \setminus
\left\{x\right\} \subset X$, which is an irreducible, noetherian spectral space
of smaller Krull dimension.\footnote{Geometrically, $X = \spec(A)$ for some
domain $A$, in which case $A$, in which case $X_x \setminus \left\{x\right\}$ corresponds to $\spec(A_{\mathfrak{p}})
\setminus \left\{\mathfrak{p}\right\}$ for some $\mathfrak{p} \in \spec(A)$.}
    Therefore, by induction, $\mathcal{G}_x \xrightarrow{\sim} \mathcal{G}(X_x \setminus
\left\{x\right\}) \xrightarrow{\sim}
\mathcal{G}_\eta$ as desired. 
\end{proof} 

\begin{corollary} 
\label{caseperfectdomain}
Let $A$ be a perfect integral domain which is ind-smooth over $\mathbb{F}_p$, and let $R$ be a smooth $A$-algebra. Then $\K(R; \mathbb{Z}_p) \to \K(R \otimes_A \mathrm{Frac}(A);
\mathbb{Z}_p)$ is an equivalence.
\end{corollary} 
\begin{proof} 
Equivalently, the assertion is that $\K(R; \mathbb{F}_p) \xrightarrow{\sim}
\K(R[1/t]; \mathbb{F}_p)$ for any nonzero $t \in A$. 
As such, we may reduce to the case where $A$ is the perfection of a smooth domain over
$\mathbb{F}_p$. 
Then $\mathrm{Spec}(A)$ is a noetherian, irreducible spectral space of finite
Krull dimension. 
We have a sheaf of spectra $\sF$ on $\mathrm{Spec}(A)$ which sends an open
subset $U \subset \mathrm{Spec}(A)$ to $\sF(U) = \K( \spec(R) \times_{\spec(A)} U;
\mathbb{F}_p)$. \Cref{modpzero} (with $d = p$) and \Cref{spectrallemma} imply that $\sF$ is a constant presheaf. 
\end{proof} 

We now return to smooth algebras over valuation rings and prove the main
results of the section.

\begin{proposition} 
\label{Gvanishmodp}
Let $V$ be a perfect valuation ring of characteristic $p$, let $t \in V$ be nonzero, and let $\overline{R}$ be a smooth $V/t$-algebra. Then $\G( \overline{R}; \mathbb{Z}_p) = 0$. 
\end{proposition} 
\begin{proof} 
We can lift $\overline{R}$ to a smooth $V$-algebra $R$ by \cite[Tag
07M8]{stacks-project}. The $V$-algebra $R$ is weakly regular stably coherent
(\Cref{smooth_over_val_are_nice}).
By \Cref{Quillendevcoherent} we have a localization sequence $\G(\overline{R}) \to \K(R) \to \K(R[1/t]) $, and the result now follows from
\Cref{caseperfectvalring} (for the valuation rings $V$ and $V[1/t]$) which shows $\K(R; \mathbb{Z}_p) \xrightarrow{\sim}
\K(R[1/t]; \mathbb{Z}_p)$. 
\end{proof} 

The next result establishes~\Cref{generalNizthm}(1).

\begin{theorem} 
\label{smoothoverperfdval}
    Let $\mathcal{O}$ be a perfectoid valuation ring and let $R$ be a smooth
    $\mathcal{O}$-algebra. Then the map 
    $\K( R; \mathbb{Z}_p) \to \K(R \otimes_{\mathcal{O}}
    \mathrm{Frac}(\mathcal{O}); \mathbb{Z}_p)$ is an equivalence. 
\end{theorem} 
\begin{proof}
    We may assume that $\mathcal{O}$ is of mixed characteristic, since the positive
    characteristic case has already been handled in \Cref{caseperfectvalring}.
    Let $t=up\in \mathcal{O}$ be a unit
    multiple of $p$ admitting
    a compatible sequence of $p$-power roots (see~\cite[Lem.~3.9]{BMS1}) and let
    $t^\flat=(t,t^{1/p},t^{1/p^2},\dots)$ be the corresponding element of the
    tilt $\mathcal{O}^{\flat}$. Note that
    $\mathrm{Frac}(\mathcal{O})=\mathcal{O}[1/p]=\mathcal{O}[1/t]$. Then
    $\mathcal{O}^{\flat}$ is a perfect valuation ring of
    characteristic $p$ and the multiplicative untilting map $\#\colon
    \mathcal{O}^\flat\to \mathcal{O}$
    induces an isomorphism of rings $\mathcal{O}^\flat/t^\flat
    \mathcal{O}^\flat \cong \mathcal{O}/p\mathcal{O}$
    \cite[(2.1.2.2)]{CS19}. So we may view $R/tR\iso R/pR$ as a smooth
    $\mathcal{O}^\flat/t^\flat
    \mathcal{O}^\flat$-algebra, whence $\G(R/tR; \mathbb{Z}_p) = 0$ by
    \Cref{Gvanishmodp}. Then the localization sequence of
    \Cref{Quillendevcoherent} completes the proof.
\end{proof} 

We also record the following strengthening of \Cref{smoothoverperfdval}.

\begin{corollary}\label{corol_t-smooth}
    Let $\mathcal{O}$ be a perfectoid valuation ring, $t\in\mathcal{O}$ a non-zero element, and
    $R$ a $t$-torsion-free $\mathcal{O}$-algebra such that
    $\mathcal{O}/t\mathcal{O}\to R/tR$ is smooth. Then
    the
    map $\K( R; \mathbb{Z}_p) \to \K(R[1/t]; \mathbb{Z}_p)$ is an
    equivalence.
\end{corollary}

\begin{proof}
    We may lift $R/tR$ to a smooth $\mathcal{O}$-algebra $R'$ by \cite[Tag
07M8]{stacks-project}. Then, by the infinitesimal lifting criterion for
    smoothness, we may lift the identification $R'/tR'=R/tR$ to a morphism
    $R'\to\widehat R$ where the hat denotes $t$-adic completion. This in turn
    induces a morphism $\widehat{R'}\to \widehat R$, which is an isomorphism
    since both sides are $t$-torsion-free and it is an isomorphism modulo $t$.
    Considering the following diagram in which the outer two squares
\[\xymatrix{
\K(R')\ar[r]\ar[d] & \K(\widehat{R'})\ar@{=}[r]\ar[d] & \K(\widehat R)\ar[d] & \K(R)\ar[l]\ar[d] \\
\K(R'[1/t])\ar[r]& \K(\widehat{R'}[1/t])\ar@{=}[r] & \K(\widehat{R}[1/t]) &\K(R[1/t])\ar[l]
}\]
    are homotopy cartesian, the problem reduces to showing that
    $\K(R';\mathbb{Z}_p)\xto\sim
    \K(R'[1/t];\mathbb{Z}_p)$. If $\mathcal{O}$ has positive characteristic then this follows from
    \Cref{caseperfectvalring} for both $\mathcal{O}$ and $\mathcal{O}[1/t]$. If
    $\mathcal{O}[1/t]=\mathcal{O}[1/p]$ then this follows from Theorem
    \ref{smoothoverperfdval}. It remains to treat the case that $\mathcal{O}$ is of mixed
    characteristic and that $t\not\in\sqrt{p\mathcal{O}}$. In this case,
    $p\mathcal{O}\subseteq p\mathcal{O}[1/t]\subseteq \mathcal{O}$, whence
    $\mathcal{O}[1/t]$ is $p$-adically complete
    and separated and thus is a perfectoid valuation ring. Therefore we conclude in
    this case by applying Theorem \ref{smoothoverperfdval} to both $\mathcal{O}$ and
    $\mathcal{O}[1/t]$.
\end{proof}

\section{Cdh sheaves on perfect schemes}
In this section we present some cdh-descent properties for localizing invariants on perfect schemes and on their Witt vectors.
We first recall the definition of the cdh-topology in the non-noetherian setting. 
Throughout, we will use the Nisnevich topology in the non-noetherian setting,
cf.~\cite[Sec.~3.7]{sag}. 

\begin{definition}[The cdh-topology]\label{def_cd_squares}
    An {\em abstract blow-up square} of schemes
    \begin{equation}\begin{gathered}  \xymatrix{
    Y' \ar[d]  \ar[r] &  X' \ar[d]^f  \\
    Y \ar[r]_i &  X 
    } \label{abssquare} \end{gathered}\end{equation}
    is a cartesian square where $i$ is a finitely presented closed embedding and
    $f$ is a proper finitely presented morphism inducing an isomorphism $X'\setminus
    Y'\stackrel\sim\to X\setminus Y$.
    The {\em cdh topology} on the category $\mathrm{Sch}^{\sub{qcqs}}$ of qcqs  schemes
    is the topology generated by the Nisnevich topology and by $\{Y\to X,\,X'\to X\}$ as one runs
    over all abstract blow-up squares of qcqs schemes. 
	We will also work with
    the restriction of these topologies to $\mathrm{Sch}_A^\sub{qcqs}$, the
    category of qcqs schemes over a base ring $A$.
\end{definition}

\begin{proposition}\label{prop_cdh_via_cd_structure}
    Let $A$ be a base ring    and let $\mathcal D$ be a complete $\infty$-category. If
    $E\colon\mathrm{Sch}_A^{\sub{qcqs},\op}\to\mathcal D$ is a Nisnevich sheaf, then the following are equivalent:
\begin{enumerate}
\item $E$ is a cdh sheaf;
\item $E$ sends all abstract blow-up squares in $\mathrm{Sch}^\sub{qcqs}_A$ to homotopy cartesian squares of $\mathcal D$;
\end{enumerate}
\end{proposition}
\begin{proof}
    Since we are already assuming that $E$ is a Nisnevich sheaf, the equivalence of
    (1) and (2) is a result of Voevodsky about
    cd-structures~\cite[Cor.~5.10]{Voe2}. We refer to~\cite[Thm.~3.2.5]{ahw1}
    for a modern treatment.
\end{proof}

The cdh sheaves of interest to us will satisfy the following excision property.

\begin{definition}[Excision]
    Given a base ring $A$ and a functor $\sF$ from the category of $A$-algebras to
    some stable $\infty$-category $\mathcal D$, we say that $\sF$
    \emph{satisfies excision} if the following holds. For every map of pairs
    $f\colon (B, I) \to (C, J)$, where $B$, $C$ are $A$-algebras and $I \subset B$, $J \subset C$ are ideals such that $f$ carries $I$ isomorphically onto $J$,
then $\sF$ carries the square (usually called a \emph{Milnor square})
\[ \xymatrix{
B \ar[d]  \ar[r] &  C \ar[d]  \\
B/I \ar[r] &  C/J,
}\]
to a homotopy cartesian square in $\mathcal D$.
\end{definition}

Given a ring $A$ (commutative as always), we will study
certain localizing invariants $\mathrm{Mod}_{\sub{Perf}(A)}(\catst)\to \mathcal D$,
where $\mathcal D$ is a stable $\infty$-category and
$\mathrm{Mod}_{\sub{Perf}(A)}(\catst)$ is the $\infty$-category of $A$-linear
stable $\infty$-categories, i.e., modules over the stably symmetric monoidal
$\infty$-category $\mathrm{Perf}(A)$ viewed as an algebra object of $\catst$.
See, for example, \cite[Rmk.~1.7]{LT19} or \cite[Sec.~3.2]{CMNN} for further details.
In this section we will be interested in the case
where $A$ is a perfect ring or its ring of Witt vectors.

Recall that an $\mathbb{F}_p$-scheme $X$ is called \emph{perfect} if the
absolute Frobenius $\phi\colon X\to X$ is an isomorphism. Given an arbitrary $\bb
F_p$-scheme $X$, its {\em perfection} is the scheme
$X_\sub{perf}:=\varprojlim_\phi X$. See \cite[Sec.~3]{bhatt-scholze-grass} for
an account of the theory of perfect schemes.

Here is the main theorem of this section, showing that localizing invariants of perfections are cdh sheaves.

\begin{theorem}\label{thm_cdhdesc}
    Let $A$ be a perfect $\mathbb{F}_p$-algebra, $\mathcal C$ a $W(A)$-linear stable
    $\infty$-category such that $\mathcal C[\tfrac1p]\simeq 0$, and
    $E\colon\Mod_{\sub{Perf}(W(A))}(\catst)\rightarrow\Dscr$ a
    localizing invariant of $W(A)$-linear stable $\infty$-categories valued in
    a stable $\infty$-category $\mathcal D$. Then there is a cdh sheaf $\mathcal
    F\colon\mathrm{Sch}^\sub{qcqs,op}_A\to\mathcal D$ characterised by \[ \mathcal
    F(\Spec B)=E(  \mathcal{C} \otimes_{\mathrm{Perf}(W(A))}
    \mathrm{Perf}(W(B_\sub{perf})) )\] for all $A$-algebras $B$; moreover,
    $\mathcal F$ satisfies excision.
\end{theorem} 
\begin{proof}
Firstly, by writing $\mathcal{C}$ as a colimit as in \cite[Prop.~2.15]{BCM20},
    we can assume that $\mathcal{C}$ is a $W_n(A)$-linear $\infty$-category for
    some $n  \gg 0$; here we implicitly use that a filtered colimit of cdh
    sheaves is a cdh sheaf, which follows from Proposition
    \ref{prop_cdh_via_cd_structure} as homotopy cartesian squares are preserved under
    filtered colimits, and since Nisnevich descent can be tested via Nisnevich
	 excision \cite[Th.~3.7.5.1]{sag}. The goal is to construct a cdh sheaf
     given on affines by  $\Spec B\mapsto E( \mathcal{C}
     \otimes_{\mathrm{Perf}(W_n(A))}
    \mathrm{Perf}(W_n(B_\sub{perf})))$; this has the advantage that it extends to qcqs
    $A$-schemes $X$ by replacing $W_n(B_\sub{perf})$ by the scheme
    $W_n(X_\sub{perf})$.

    We may moreover replace $E$ by $E(\mathcal{C}
    \otimes_{\mathrm{Perf}(W_n(A))}-)$ so that $E$ is now a localizing
    invariant of $\mathrm{Perf}(W_n(A))$-linear $\infty$-categories and our
    goal is to show simply that the
	 functor \[\mathrm{Sch}_A^\sub{qcqs}\to\mathcal D,\quad
    X\mapsto E(W_n(X_\sub{perf})):=E(\textrm{Perf}(W_n(X_\sub{perf})))\] is a cdh sheaf which satisfies excision.

The construction 
$X \mapsto E( W_n(X_{\mathrm{perf}}))$ satisfies Nisnevich
descent by a result of Thomason--Trobaugh \cite{TT90},
cf.~\cite[App.~A]{CMNN}. 
To verify cdh-descent, we need to show that any abstract blow-up square
\eqref{abssquare}

We prove this using Bhatt--Scholze's $v$-descent for quasi-coherent complexes on
perfect schemes \cite[Cor.~11.28]{bhatt-scholze-grass}. Since $Y\to X$ and $X'\to X$ are finitely presented, {\em loc.~cit.} implies that the square 
\[\xymatrix{
\mathrm{QCoh}(W_n(X_\sub{perf})) \ar[d]  \ar[r] &  \mathrm{QCoh}(W_n(X'_\sub{perf})) \ar[d]_{j^*}  \\
\mathrm{QCoh}(W_n(Y_\sub{perf})) \ar[r] &  \mathrm{QCoh}(W_n(Y'_\sub{perf})) 
}\]
    is a pullback of $\infty$-categories, where $j^*$ is pullback along
    $Y'_{\sub{perf}}\rightarrow X'_{\sub{perf}}$. Furthermore, the right adjoint $j_*$ of $j^*$ is fully faithful: indeed, $j^*j_*\simeq\mathrm{id}$ since $C\otimes^L_BC\simeq C$ for any surjection of perfect rings $B\to C$ \cite[Lem~3.16]{bhatt-scholze-grass}. That is, the square
\[\xymatrix{
\mathrm{Perf}(W_n(X_\sub{perf})) \ar[d]  \ar[r] &  \mathrm{Perf}(W_n(X'_\sub{perf})) \ar[d]  \\
\mathrm{Perf}(W_n(Y_\sub{perf})) \ar[r] &  \mathrm{Perf}(W_n(Y'_\sub{perf})) 
}\]
in $\mathrm{Mod}_{\sub{Perf}(W_n(A))}(\catst)$ is excisive in the sense of
    \cite[Def.~14]{tamme-excision} (note that all our schemes are qcqs, so $\mathrm{Ind}(\mathrm{Perf}(-))=\mathrm{QCoh}(-)$), and so Tamme's excision criterion \cite[Thm.~18]{tamme-excision} shows that 
\[\xymatrix{
E(W_n(X_\sub{perf})) \ar[d]  \ar[r] &  E(W_n(X'_\sub{perf})) \ar[d]  \\
E(W_n(Y_\sub{perf})) \ar[r] &  E(W_n(Y'_\sub{perf})) 
}\]
is indeed homotopy cartesian, as required to complete the proof of cdh-descent.

 Given an excision situation of
    $A$-algebras, $(B, I) \to (C, J)$, then
    $(B_\sub{perf},I':=\sqrt{IB_\sub{perf}})\to
    (C_\sub{perf},J':=\sqrt{JC_\sub{perf}})$ is an
    excision situation (note that the radical $I'$ is precisely the kernel of
    $B_\sub{perf}\to (B/I)_\sub{perf}$, and similarly for $J'$), and then so is
    $(W_n(B_\sub{perf}), W_n(I')) \to (W_n(C_\sub{perf}), W_n(J'))$. We claim
    that $W_n(B_\sub{perf})\to W_n(B_\sub{perf})/W_n(I')=W_n((B/I)_\sub{perf})$
    is Tor-unital in the sense of \cite[Def.~21]{tamme-excision}. Indeed, to
    show that the canonical map
    $W_n((B/I)_\sub{perf})\otimes_{W_n(B_\sub{perf})}^LW_n((B/I)_\sub{perf})\to
    W_n((B/I)_\sub{perf})$ is an equivalence it is enough to check after base
    change along $\mathbb Z/p^n\to\mathbb Z/p$, as which point we obtain the
    equivalence $(B/I)_\sub{perf}\otimes_{B_\sub{perf}}^L(B/I)_\sub{perf}\simeq
    (B/I)_\sub{perf}$ of \cite[Lem~3.16]{bhatt-scholze-grass}. It now follows from Tamme's excision condition \cite[Thm.~28]{tamme-excision} that
\[\xymatrix{
E(W_n(B_\sub{perf}))\ar[r]\ar[d] & E(W_n(C_\sub{perf}))\ar[d]\\
E(W_n((B/I)_\sub{perf}))\ar[r] & E(W_n((C/J)_\sub{perf}))\\
}\]
is indeed homotopy cartesian, as desired.
\end{proof}

\begin{question}
    Given an abstract blow-up square of qcqs $\bb F_p$-schemes as in
    \Cref{def_cd_squares} but without the assumption that $f\colon X' \to X$ is
	 finitely presented, is it true that applying
    $\mathrm{QCoh}(-_\sub{perf})$ gives a pullback square of
    $\infty$-categories?
	Under the additional assumption that $f:X'\to X$ is
    finitely presented this is precisely
    \cite[Cor.~11.28]{bhatt-scholze-grass}, which was used above. 
	\end{question}

The above result will be useful in reducing questions to henselian valuation
rings in light of the next result, which we quote for convenience.
We refer to \cite[Sec.~2.3]{EHIK} for an introduction to Jaffard's
notion of valuative dimension; note that the perfection of any finitely
generated $\mathbb F_p$-algebra has finite valuative dimension, given by its
Krull dimension (indeed, there is a one-to-one correspondence between the
valuation subrings of a field of characteristic $p$ and of its perfection, or
else more generally we could apply~\cite[Prop.~4, p.~54]{Jaffard}).
A cdh-sheaf on $\mathrm{Sch}_A^\sub{qcqs}$ is said to be \emph{finitary} if it preserves
filtered colimits of $A$-algebras. 

\begin{proposition}[{\cite[Cor.~2.4.19]{EHIK}}]\label{prop_finite_dimension}
If $A$ is an $\bb F_p$-algebra of finite valuative dimension, then
a map of finitary cdh-sheaves $\mathrm{Sch}_A^\sub{qcqs,op} \to \mathcal{S}$
(for $\mathcal{S}$ the $\infty$-category of spaces) which is
an equivalence on henselian valuation rings is an equivalence. 
\end{proposition}

\section{Applications}
In this section, we complete the proofs of our main theorems concerning
algebraic $\K$-theory.

\subsection{Characteristic $p$}\label{ss_p}
We begin with two propositions about the $\K$-theory of perfect schemes; the results in mixed characteristic will then follow by an elaboration of the arguments. 

\begin{proposition} 
\label{KisKHperfect}
Let $R$ be a smooth algebra over a perfect $\mathbb{F}_p$-algebra $A$. Then the canonical map
$\K(R)\to \KH(R)$ is an equivalence.
\end{proposition} 
\begin{proof} 
Taking a filtered colimit, it suffices to prove the result when $A$ is the perfection of a finitely generated $\mathbb F_p$-algebra.
\Cref{thm_cdhdesc}, with $\mathcal C=\mathrm{Perf}(R)$ and $\mathrm{Perf}(R[T])$
respectively, implies that the functors $X\mapsto \K(X_\sub{perf}\otimes_AR)$ and
$X\mapsto \K(X_\sub{perf}\otimes_AR[T])$, from finitely presented $A$-schemes to
spectra, are cdh sheaves; recall here that
$\mathrm{Perf}(X_\sub{perf})\otimes_{\mathrm{Perf}(A)}\mathrm{Perf}(R)\simeq
\mathrm{Perf}(X_\sub{perf}\otimes_AR)$ and similarly for $R[T]$, by e.g.,
\cite[Th.~4.8.4.6]{HA}. To check that the natural map $\K(R)\to \K(R[T])$ is an
equivalence, we therefore reduce by \Cref{prop_finite_dimension} to proving that
$\K(V_\sub{perf}\otimes_AR)\xto\sim \K(V_\sub{perf}\otimes_AR[T])$ for all henselian valuation rings $V$ under $A$; but that is a special case of \Cref{stablycohfflat}(2) since $V_\sub{perf}$ is a valuation ring.
\end{proof} 

\begin{question}[Cartier smooth algebras] 
Can \Cref{KisKHperfect} be extended to those $\mathbb{F}_p$-algebras $R$ which are
Cartier smooth in the sense of \cite{KM21}, i.e., those for which the cotangent
complex $L_{R/\mathbb{F}_p}$ is a flat $R$-module and the Cartier isomorphism
for de Rham cohomology holds?
\end{question} 

\begin{question}[Motivic refinement]
    Let $R$ be as in Proposition \ref{KisKHperfect}, or more generally Cartier
    smooth as in the previous question. Then the Geisser--Levine theorem 
    \cite{GL00}
    holds
    Zariski locally on $\Spec R$ by \cite{KM21}, and so the Zariski local Postnikov
    filtration on $\K(R;\bb F_p)$ deserves to be termed the ``motivic filtration''.
    Does the equivalence $\K(R;\bb F_p)\quis \K(R[T];\bb F_p)$ (following from
    Proposition \ref{KisKHperfect}) upgrade to a filtered equivalence, i.e.,
    are the canonical maps $R\Gamma_\sub{Zar}(\Spec R,\Omega^i_{\sub{log}})\to
    R\Gamma_\sub{Zar}(\Spec R[T],\Omega^i_{\sub{log}})$ equivalences?
    When $A$ is a perfect valuation ring, the answer is `yes': perfect
    valuation rings are ind-smooth \cite{Temkin13}, so the motivic filtration defined at the
    top of~\Cref{motivic} is $\AA^1$-invariant.
\end{question}

Regarding the $\K$-theory of perfect schemes themselves (rather than smooth
schemes over them), we record the following calculation since it has not
explicitly appeared previously; the cdarc topology, which is a completely
decomposed analogue of the arc topology \cite{BM21}, is defined in \cite{EHIK}.

\begin{corollary}
\begin{enumerate}
\item For any perfect qcqs $\bb F_p$-scheme $X$, there is a natural equivalence $\K(X;\mathbb Z/p^r)\simeq    R\Gamma_\sub{cdh}(X,\mathbb Z/p^r)$, where the right side denotes
    cdh cohomology on $\mathrm{Sch}^\sub{qcqs}_X$ of the constant sheaf $\mathbb
    Z/p^r$.
\item The presheaf $X\mapsto \K(X_\sub{perf};\mathbb Z/p^r)$ on
    $\Sch^{\sub{qcqs,op}}_{\mathbb F_p}$ satisfies cdarc descent.
\end{enumerate}
\end{corollary}

\begin{proof}
    Theorem~\ref{thm_cdhdesc} implies that $\K((-)_\sub{perf};\mathbb Z/p^r)$ is a
    cdh sheaf on $\Sch^{\sub{qcqs,op}}_{\mathbb F_p}$. For any qcqs $\bb F_p$-scheme $X$ there is a map
    $\bb Z/p^r=\K(\mathbb F_p;\mathbb Z/p^r)\rightarrow\K(X_\sub{perf};\mathbb Z/p^r)$, which
    induces a map $R \Gamma_{\mathrm{cdh}}(-, \mathbb{Z}/p^r) \rightarrow\K((-)_\sub{perf};\mathbb Z/p^r)$ of cdh
    sheaves. 
	 
	 Note that both these sheaves commute with filtered colimits of rings, the
	 latter because $\K$-theory is
	 finitary and the former by \cite[Tag
0737]{stacks-project}; furthermore, both these sheaves take values in
$\mathcal{D}(\mathbb Z/p^r)_{\leq 0}$, the latter by Theorem~\ref{QHKthm}. Since both these cdh sheaves satisfy Milnor excision (the latter by \Cref{thm_cdhdesc}), ``(1) implies (3)'' of the main theorem of~\cite{EHIK} therefore implies that they are both in fact cdarc sheaves. This completes the proof of part (2).

We have also reduced part (1) to the case that $X=\Spec A$ is the spectrum of a
perfect $\bb F_p$-algebra; writing $A$ as a filtered colimit of perfections of
finite type $\bb F_p$-algebras then allows us to even assume that $A$ has finite
valuative dimension. Then viewing $R\Gamma_\sub{cdh}(-,\bb Z/p^r)\to  \K((-)_\sub{perf};\mathbb
Z/p^r)$ as a map of cdh sheaves on $\mathrm{Sch}^\sub{qcqs}_A$,
\Cref{prop_finite_dimension} reduces the problem to checking that $\bb Z/p^r\to
\K(V_\sub{perf};\mathbb Z/p^r)$ is an equivalence for all henselian valuation rings $V$ under $A$. But $V_\sub{perf}$ is again a valuation ring, hence weakly regular
    stably coherent by \Cref{valstcoh} and so has no negative $\K$-groups by
	 \Cref{stablycohfflat}; then \Cref{QHKthm} shows that indeed $\bb
	 Z/p^r\stackrel{\simeq}\to \K(V_\sub{perf};\mathbb Z/p^r)$, as desired.
\end{proof}

In~\cite[Corollary~5.6]{bhatt-scholze-grass}, Bhatt and Scholze prove that if
$A$ is the perfection of a regular $\FF_p$-algebra, then  there is a
localization fiber sequence \[ \K(A) \to \K(W(A)) \to \K(W(A)[1/p])  .\]  More
precisely, they show that the natural map $\K(A)\rightarrow \K(\W(A)\text{ on
}pW(A))$ is an equivalence; this is used to define their determinant line bundle
on the Witt vector affine Grassmannian. We will prove more generally that this
assertion is true for any perfect $\FF_p$-algebra $A$, and even for algebras
over $W(A)$ satisfying a ``$p$-smoothness'' condition as in
\Cref{corol_t-smooth}.

\begin{proposition} 
\label{Kdevperfect}
Let $A$ be any perfect $\mathbb{F}_p$-algebra, and let $R$ be a $p$-torsion-free $W(A)$-algebra such that $A\to R/pR$ is smooth. Then the natural map $\K(R/pR)\to \K(R\textrm{ on }pR)$ is an equivalence, so there is a localization sequence $\K(R/pR) \to \K(R) \to \K(R[1/p])$.
\end{proposition} 
\begin{proof}
Since we can replace $R$ by its $p$-completion, it suffices to assume that $R$
is $p$-complete. 
The functor that sends the perfect $\mathbb{F}_p$-algebra $A$ to 
isomorphism classes of $p$-complete $W(A)$-algebras $R$ with $R/p$ smooth over
$A$ commutes with filtered colimits in $A$, since it is also isomorphic to the
functor of smooth $A$-algebras. 
Using this observation, and the fact that 
$\K(R \text{ on } pR)$ commutes with filtered colimits in $p$-complete $R$, 
we may reduce to the case where $A$ is the perfection of a finitely generated
$\mathbb{F}_p$-algebra.  

\Cref{thm_cdhdesc}, with $\mathcal C=\mathrm{Perf}(R/pR)$ and
$\mathrm{Perf}(R\text{ on }pR)$ respectively, implies that there are cdh sheaves
on $\mathrm{Sch}^\sub{qcqs}_A$ given on affines by $\Spec B\mapsto
\K(B_\sub{perf}\otimes_AR/pR)$ and $\K(W(B_\sub{perf})\otimes_{W(A)}R\text{ on
    }p)$.\footnote{Here, and on occasion below, we break with our
    convention and replace for example $\K(W(B_\sub{perf})\otimes_{W(A)}R\text{ on
}p(W(B_\sub{perf})\otimes_{W(A)}R))$ by $\K(W(B_\sub{perf})\otimes_{W(A)}R\text{ on
    }p)$ for readability.} To prove that the natural map $\K(R/pR)\to \K(R\textrm{ on }pR)$ is an
equivalence, we therefore reduce by \Cref{prop_finite_dimension} to proving that
$\K(V_\sub{perf}\otimes_AR/pR)\xto\sim \K(W(V_\sub{perf})\otimes_{W(A)}R\text{ on
}p)$ for all henselian valuation rings $V$ under $A$. This follows from
\Cref{locregdivisor} with $t = p$. 
\end{proof}

\begin{corollary} 
\label{KisKHpadicWittperfect}
Let $A$ and $R$ be as in the statement of \Cref{Kdevperfect}. 
\begin{enumerate}
\item There is a localization sequence $\KH(R/pR)\to\KH(R)\to\KH(R[\tfrac1p])$.
\item The natural map $\K(R; \mathbb{Z}_p) \to \KH(R; \mathbb{Z}_p)$ is an equivalence. 
\end{enumerate}
\end{corollary} 
\begin{proof} 
Part (1) follows by applying \Cref{Kdevperfect} to $R[T_0,\dots,T_n]$ for all $n\ge0$ and taking the geometric realisation.

Part (2) then follows by comparing the localization sequences in $\K(-;\mathbb
Z_p)$ and $\KH(-;\mathbb Z_p)$. Indeed we know that $\K(R/pR) \xrightarrow{\sim} \KH(R/pR)$ by \Cref{KisKHperfect}, while $\K(R[1/p]; \mathbb{Z}_p) \to \KH(R[1/p]; \mathbb{Z}_p)$ is an
equivalence thanks to \cite[Prop.~1.6]{WeibelKH}.
\end{proof} 

\begin{remark} 

\Cref{KisKHperfect} and \Cref{Kdevperfect} belong to a general class of results: for many
purposes, perfect rings and their rings of Witt vectors behave similarly to regular rings. 
For example, in \cite[Prop.~11.31]{bhatt-scholze-grass}, it is shown that perfectly
finitely presented perfect $\mathbb{F}_p$-algebras have finite global dimension.
The above results are indications of a similar phenomenon in algebraic
$\K$-theory. Note however that perfect rings can have nontrivial nonconnective
$\K$-theory.
\end{remark} 

The next result refines the classical \Cref{QHKthm}.

\begin{proposition} 
\label{perfectFpvanish}
Let $R$ be a smooth algebra of relative dimension $\le d$ over a perfect
$\mathbb{F}_p$-algebra. Then $\K(R; \mathbb{F}_p)$ is $d$-truncated. 
\end{proposition} 
\begin{proof} 
We may assume by passage to a filtered colimit that the perfect base algebra
    $A$ is the perfection of a finitely generated $\mathbb{F}_p$-algebra. In
    this case, the Zariski topos of $R$ is hypercomplete since $\spec(R)$ has
    finite Krull dimension and is noetherian, \cite[Cor.~7.2.4.20]{HTT}.  We
    therefore reduce to showing that $\K(R_\mathfrak p;\mathbb F_p)$ is
    $d$-truncated for each prime ideal $\mathfrak p\subset R$. This may be
    shown in two ways. Either one has by \cite[Th.~2.1]{KM21} the analog of the Geisser--Levine
theorem \cite{GL00}, computing the mod-$p$ $\K$-groups of $R_\mathfrak p$ as the
    module of logarithmic differential forms, i.e., $\K_n(R_\mathfrak p;
    \mathbb{F}_p) \simeq \Omega^{n}_{R_\mathfrak p,\mathrm{log}}$ for $n\ge 0$;
    but these vanish when $n>d$. Alternatively, we could avoid
    \cite[Th.~2.1]{KM21} by instead using \Cref{thm_cdhdesc} and
    \Cref{prop_finite_dimension} to reduce to the case that $A$ is a perfect valuation
    ring, hence ind-smooth \cite{Temkin13}, and apply the usual Geisser--Levine
    theorem.
\end{proof} 

\subsection{Mixed characteristic}
Next we treat the mixed characteristic results.
Given a perfectoid valuation ring $\mathcal{O}$ of mixed characteristic we
briefly recall the tilting correspondence, for which we refer to
\cite[Sec.~3]{BMS1} and \cite[Sec.~2.1]{CS19} for further details. As in the
proof of \Cref{smoothoverperfdval} we may rescale $p$ by a unit so that it
admits a compatible sequence of $p$-power roots, and we let $t^\flat$ be the
corresponding of the tilt $\roi^\flat$ so that the untilting map induces an
isomorphism $\roi^\flat/t^\flat\xto\sim \roi/p$. For a
perfect $\mathcal{O}^\flat$-algebra $A$, we let $A^{\sharp} =
W(A)\otimes_{W(\mathcal{O}^\flat),\theta} \mathcal{O}$ denote its
\emph{untilt}. The untilt depends only on the (classical) $t^\flat$-adic
completion of $A$, and $A\mapsto A^\sharp$ establishes an equivalence between
$t^\flat$-adically complete and separated perfect $\mathcal{O}^\flat$-algebras
and perfectoid $\mathcal{O}$-algebras; this restricts to an equivalence between
perfect $\mathcal{O}^{\flat}$-algebras in which $t^\flat=0$ and perfectoid
(equivalently, perfect) $\mathcal{O}$-algebras in which $p = 0$.

The following result exemplifies the approach through which we may study localizing invariants over a mixed characteristic perfectoid valuation ring via the cdh topology over its tilt.

\begin{proposition}\label{generalreduction}
Let $E$ a localizing invariant of $H\mathbb Z$-linear stable $\infty$-categories valued in spectra, and fix $d,m\in\mathbb Z$; make the following
assumption:

\begin{quote}
For every perfectoid valuation ring $V$ and every smooth
$V$-algebra $R_V$ such that $V/pV\to R_V/pR_V$ has relative dimension $\le d$, the spectrum
$E(R_V\text{ on }pR_V)$ is $m$-truncated.
\end{quote}
Then, for every perfectoid valuation ring $\roi$, every perfectoid
    $\roi$-algebra $A$, and every smooth $A$-algebra $R$ such that $A/pA\to
    R/pR$ has relative dimension $\le d$, the spectrum $E(R\text{
        on }pR)$ is $m$-truncated.
\end{proposition} 
\begin{proof} 
We treat a series of cases, culminating in a complete proof. For $\roi$, $A$,
    $R$ as in the statement of the proposition, \Cref{thm_cdhdesc} with
    $\mathcal C=\mathrm{Perf}(R\text{ on }pR)$ implies that there exists a cdh
    sheaf $\mathcal{F}_{\mathcal{O},A,R}$ on
    $\mathrm{Sch}^\sub{qcqs}_{A^\flat}$ given on affines by \[\Spec B\mapsto
    E(\mathrm{Perf}(R\text{ on }p R)
    \otimes_{\mathrm{Perf}(W(A^\flat))}\mathrm{Perf}(W(B_\sub{perf})))=
    E(B_\sub{perf}^\#\otimes_AR\text{ on }p),\] and that this sheaf satisfies
    excision.

{\em Case 1}: $p=0$ in $A$, i.e., $A$ is a perfect $\bb F_p$-algebra. Replacing
$\roi$ by $\roi/\sqrt{p\roi}$ we may clearly also suppose that $p=0$ in $\roi$,
i.e., that $\roi$ is a perfect valuation ring over $\bb F_p$. The assertion to
be proved is that $E(R)$ is $m$-truncated. Writing $\roi$ as a filtered colimit
of perfect valuation rings $\roi_i$ of finite rank, and then writing $A$ as a
filtered colimit of perfectly finitely presented $\roi_i$-algebras for varying
$i$, we reduce to the case that $\roi$ has finite rank and $A$ is a perfectly
finitely presented $\roi$-algebra. Therefore $A$ has finite valuative dimension
\cite[Corol.~2.3.3]{EHIK} and \Cref{prop_finite_dimension}  applies to the cdh
sheaf $\Omega^{m+1}\Omega^{\infty }\mathcal{F}_{\roi,A,R}$: so the desired
$m$-truncatedness of $E(R)$ follows from the assumed $m$-truncatedness of
$E(V_\sub{perf}\otimes_AR)$, as $V$ varies over henselian valuation rings
under $A$.

{\em Case 2}: $\roi$ is mixed characteristic of rank one, and $A$ is $p$-torsion-free. Similarly to the previous case, we begin by reducing to a finitely presented situation. Write the tilt $A^{\flat}$ as  a filtered colimit of perfectly finitely presented
$\mathcal{O}^{\flat}$-algebras $B_i^{\flat}$, so that $A$ is the
(classical) $p$-completion of the filtered colimit $\colim_iB_i^\sharp$ of the
perfectoid $\mathcal{O}$-algebras $B_i^\sharp$. Since $\colim_iB_i^\sharp$ is
$p$-torsion-free (otherwise its $p$-completion $A$ would also contain
$p$-torsion; this implicitly uses that the $p$-power torsion in any perfectoid
is killed by $p$, so that $\colim_iB_i^\sharp$ has bounded $p$-power torsion),
we may argue just as at the beginning of the proof of \Cref{corol_t-smooth} to
find a smooth $\colim_iB_i^\sharp$-algebra $R'$ equipped with an isomorphism
$\widehat{R'}\cong \widehat R$ of $A$-algebras, where the hats denote $p$-adic
completions. Since $\mathrm{Perf}(R'\text{ on }pR')\simeq
    \mathrm{Perf}(\widehat{R'}\text{ on }p\widehat{R'})\simeq \mathrm{Perf}(\widehat{R}\text{ on
    }p\widehat{R})\simeq \mathrm{Perf}(\widehat{R}\text{ on }p\widehat{R})$, the first and third
isomorphisms being \cite[Lem.~5.12]{BhattTannaka}, the problem reduces to showing that $E(R'\text{ on }pR')$ is
    $m$-truncated. Descending $R'$ to a smooth algebra over $B_i^\sharp$ for
    $i\gg0$ and taking the filtered colimit, we may henceforth assume in this
    case that $A$ is the untilt of a perfectly finitely presented
    $\roi^\flat$-algebra $B$.\footnote{$A^\flat$ is not in general equal to
    $B$, but rather to its $t^\flat$-adic completion in the notation of the
    opening paragraph of the subsection. We also remark that by replacing $A$
    by $B_i^\sharp$ in this fashion might mean that $A$ is no longer
    $p$-torsion-free, but this condition is not required in the remainder of
    the proof.} Since $\roi^\flat$ has finite rank (even rank one), therefore
    $B$ has finite valuative dimension; then the $m$-truncatedness of
    $E(R\text{ on }pR)$ follows as in Case 1, namely by applying
    \Cref{prop_finite_dimension} and using the hypothesis that
    $E(V_\sub{perf}^\#\otimes_AR\text{ on }p)$ is $m$-truncated as $V$ varies over henselian valuation rings under $A$.

{\em Case 3}: $\roi$ is mixed characteristic of rank one, but no conditions on $A$. The perfectoid ring $A$ fits into a Milnor square with surjective arrows
\[\xymatrix{
A\ar[r]\ar[d] & \overline A\ar[d] \\
A_0\ar[r] & \overline A_0
}\]
where $\overline A$ is a $p$-torsion-free perfectoid ring, and $A_0$ and
    $\overline A_0$ are perfect $\mathbb F_p$-algebras \cite[2.1.3]{CS19};
    tilting each term forms a Milnor square of perfect $\mathbb F_p$-algebras
    [loc.~cit]. The sheaf $\mathcal{F}_{\roi,A,R}$ satisfies excision, so the
    square
\[\xymatrix{
    E(R\text{ on }pR)\ar[r]\ar[d] & E(\overline A\otimes_AR\text{ on
    }p)\ar[d] \\
E(A_0\otimes_AR)\ar[r] & E(\overline A_0\otimes_AR)
}\]
is homotopy cartesian. The top right term is $m$-truncated by case 2, and the bottom two terms are $m$-truncated by case 1. Therefore the top left term is $m$-truncated.

{\em Case 4}: The general case. If $p=0$ in $\roi$ then we may appeal to case $1$. Otherwise, similarly to case 3, we use an excision trick, this time applied to the Milnor square of valuation rings
\[\xymatrix{
\roi\ar[r]\ar[d] & \roi_{\frak p}\ar[d]\\
\roi/\frak p\ar[r] & k(\frak p)
}\]
where $\frak p=\sqrt{p\roi}$. Note that $\roi_{\frak p}$ is a perfectoid valuation ring of mixed characteristic of rank one (argue as at the end of \Cref{corol_t-smooth} to see $p$-completeness), and the bottom two terms are perfect valuation rings. Base changing along $\roi\to A$, we claim that 
\[\xymatrix{
A\ar[r]\ar[d] & A\otimes_\roi\roi_{\frak p}\ar[d]\\
A\otimes_\roi\roi/\frak p\ar[r] & A\otimes_\roi k(\frak p)
}\]
is an Milnor square of perfectoid rings. Firstly, the terms are perfectoid rings
after $p$-adic completion by \cite[Prop.~2.1.11(2)]{CS19}; so the bottom two
terms, which are $\bb F_p$-algebras, are perfect. But $A\otimes_\roi\roi_{\frak
p}$ is already $p$-adically complete since $p(A\otimes_\roi\roi_{\frak
p})\subseteq A$ and $A$ is $p$-adically complete; so $A\otimes_\roi\roi_{\frak
p}$ is perfectoid. Secondly, 
the square is a Milnor square because it is a base-change of a Milnor square
and 
since $A \otimes_{\mathcal{O}}^{\mathbb{L}} \mathcal{O}/\mathfrak{p}$ is
discrete.\footnote{In fact, if $B_0, B_1, B_2$ are perfectoid rings with maps
$B_0 \to B_1, B_0 \to B_2$, the derived tensor product $B_1
\otimes_{B_0}^{\mathbb{L}} B_1$ has discrete derived $p$-completion; in fact,
this follows from the result for perfect $\mathbb{F}_p$-algebras
\cite[Lem.~3.16]{bhatt-scholze-grass}, which implies the analogous results for
their Wit vectors, and \cite[Lem.~3.13]{BMS1}.} 
Tilting each term forms a Milnor square of perfect $\mathbb F_p$-algebras
\cite[Prop.~2.1.4]{CS19}. The sheaf $\mathcal{F}_{\roi,A,R}$ satisfies excision, so the square
\[\xymatrix{
    E(R\text{ on }pR)\ar[r]\ar[d] & E(\overline A\otimes_AR\text{ on }p)\ar[d] \\
E(A_0\otimes_AR)\ar[r] & E(\overline A_0\otimes_AR)
}\]
is homotopy cartesian. The top right term is $m$-truncated by case 2, and the
bottom two terms are $m$-truncated by case 1. Therefore the top left term is
$m$-truncated.
\end{proof}

\begin{theorem} 
    \label{Invertingp}
Let $\roi$ be a perfectoid valuation ring, $A$ a perfectoid $\mathcal{O}$-algebra, and $R$ a smooth $A$-algebra such that $A/pA\to R/pR$ has relative dimension $\le d$. Then $\K(R; \mathbb{F}_p) \to \K(R[1/p]; \mathbb F_p)$ is $d$-truncated.
\end{theorem} 
\begin{proof}
We are claiming that $\K(R \text{ on } pR ; \mathbb F_p)$ is $d$-truncated. It
    suffices to check that the hypothesis of \Cref{generalreduction} are
    satisfied, namely that $\K(R_V \text{ on } pR_V ; \mathbb F_p)$ is
    $d$-truncated whenever $R_V$ is a smooth algebra over a perfectoid
    valuation ring $V$ such that $V/pV\to R_V/pR_V$ has relative dimension $\le
    d$. This follows from \Cref{smoothoverperfdval} (in the case of mixed
    characteristic $V$) and \Cref{perfectFpvanish} (in the case of $V$ of
    characteristic $p$).
\end{proof} 

\begin{lemma} 
\label{KisKHvalring}
Let $V$ be a valuation ring and $R$ a smooth $V$-algebra. Then, for any $t\in
    V$, the map $\K(R \text{ on } tR) \to \KH(R \text{ on } tR)$
    is an equivalence.
\end{lemma} 
\begin{proof} 
By comparing $R$ and $R[1/t]$, we see that it suffices to prove the stronger statement that $\K(R)\xrightarrow{\sim} \KH(R)$ and $\K(R[\tfrac1t])\xrightarrow{\sim} \KH(R[\tfrac1t])$. This follows from \Cref{smooth_over_val_are_nice} and \Cref{stablycohfflat}. 
\end{proof} 

\begin{theorem} 
    \label{KisKHperfectoid}
Let $\roi$ be a perfectoid valuation ring, $A$ a perfectoid $\mathcal{O}$-algebra, and $R$ a smooth $A$-algebra. Then the canonical map $\K(R; \mathbb{Z}_p) \to \KH(R; \mathbb{Z}_p)$ is an equivalence.
\end{theorem}
\begin{proof} 
We consider the spectra-valued localizing invariant $\mathrm{NK}$ of $H\mathbb Z$-linear stable $\infty$-categories
\[\mathrm{NK}\colon\mathcal
C\mapsto\mathrm{hocofib}(\K(\mathcal C;\mathbb F_p)\to \K(\mathcal
    C\otimes_{\mathrm{Perf}(\mathbb Z)}\mathrm{Perf}(\mathbb Z[T]);\mathbb
    F_p).\] \Cref{KisKHvalring} (with $t=p$) shows that $\mathrm{NK}(V\text{ on
    }pV)$ vanishes on any valuation ring $V$, whence we may apply
    \Cref{generalreduction} with any value of $m$ and so deduce that
    $\mathrm{NK}(R\text{ on }pR)$ vanishes. In other words, the square
\[ \xymatrix{
\K(R; \mathbb{Z}_p) \ar[r] \ar[d]  &  \K(R[T]; \mathbb{Z}_p) \ar[d]  \\
\K(R[\tfrac1p]; \mathbb{Z}_p) \ar[r] &  \K(R[\tfrac1p][T]; \mathbb{Z}_p)
}\]
is homotopy cartesian. The bottom horizontal arrow is an equivalence since $\K(-;\bb Z_p)$ is homotopy invariant on $\bb Z[\tfrac1p]$-algebras \cite[Prop.~1.6]{WeibelKH}, so the top horizontal arrow is also an equivalence.
\end{proof} 

\begin{remark}[The $p$-smooth case]\label{rem_p_smooth}
Similarly to \Cref{corol_t-smooth}, Theorems \ref{Invertingp} and \ref{KisKHperfectoid} remain true more generally if the $A$-algebra $R$ is merely required to be $p$-smooth rather than smooth. Here we say that an $A$-algebra $R$ is {\em $p$-smooth} if it satisfies the following equivalent conditions:
\begin{enumerate}
\item there exists a smooth $A$-algebra $R'$ such that the derived $p$-adic completions\footnote{For the sake of clarify, we remark that the derived $p$-completion of an $A$-module $M$ is $\holim_nM/^{\mathbb L}p^n$ where $M/^{\mathbb L}p^n:=M\otimes_{\mathbb Z[T],T\mapsto p^n}^{\mathbb L}\mathbb Z$ is the Koszul complex associated to multiplication by $p^n$} of $R$ and $R'$ are equivalent as animated $A$-algebras;
\item the $A/pA$-algebra $R\otimes_A^{\mathbb L}A/pA$ is discrete and smooth.
\end{enumerate}

To obtain the stronger versions of the theorems from the smooth versions, we
    use characterisation (1) and the equivalence $\mathrm{Perf}(R\text{ on
    }pR)\stackrel{\sim}\to\mathrm{Perf}(\widehat R\text{ on }p\widehat{R})$, resulting from
    the equivalence $\widehat R\otimes^{\mathbb L}_RR/pR\simeq R/pR$ and a
    strong form of Thomason's excision theorem
    \cite[Lem.~5.12(2)]{BhattTannaka} (and similarly for $R'$ in place of $R$);
    here hats denote derived $p$-adic completions. In the case of Theorem
    \ref{KisKHperfectoid} we also use the analogous equivalences for $R[T]$ and
    $R'[T]$, as well as the fact that $\K(-;\mathbb Z_p)$ is homotopy invariant
    on $\mathbb Z[\tfrac1p]$-algebras.

Lacking a simple reference, we explain why conditions (1) and (2) are indeed
equivalent.\footnote{The proof works over any base ring $A$, though we remark
that if its $p$-power torsion is not bounded then the derived $p$-completions
appearing in (1) might not be discrete.} It is clear that $(1)$ implies $(2)$,
so suppose that $R$ is an $A$-algebra satisfying condition (2). Let $R'$ be any
smooth $A$-algebra lifting the smooth $A/pA$-algebra $R\otimes_A^{\mathbb
L}A/pA$ \cite[Tag 07M8]{stacks-project}. By derived deformation theory
(cf.~\cite[Sec.~3.4]{LurieDAG}) the canonical map $R'\to R'/pR'=R\otimes_A^{\mathbb L}A/pA=\widehat{R}\otimes_A^{\mathbb L}A/pA$ may be lifted to a morphism $R'\to \widehat{R}$, which induces a morphism $\widehat{R'}\to \widehat{R}$. The fibre $F$ of the latter morphism satisfies $F\otimes_A^{\mathbb L}A/pA\simeq0$; therefore $F\otimes_A^{\mathbb L}A/^{\mathbb L}p\simeq0$ (as $A/^{\mathbb L}p$ is a bounded complex with cohomologies being $A/pA$-modules), and so $F\simeq 0$ since it is derived $p$-complete.
\end{remark}

\begin{question}
Do Theorems \ref{Invertingp} and \ref{KisKHperfectoid} hold for any perfectoid ring $A$, without assuming that it is an algebra over a perfectoid valuation ring $\roi$? This assumption appeared in the proof of case 2 of \Cref{generalreduction}, where it was used to reduce to a situation of finite valuative dimension. Indeed, the theorems therefore remain true for any perfectoid ring $A$ which receives a map $A_0\to A$ from another perfectoid ring $A_0$ having the property that $A_0^\flat$ has finite valuative dimension.

\end{question}

\appendix

\section{Localization sequences via $t$-structures}

In the appendix, we explain how to deduce the d\'evissage results of Section 2
(and slight generalizations)
from 
the  theorem of the heart due to Barwick \cite{Bar15} and
Antieau--Gepner--Heller \cite{AGH}.  
Let $\mathcal{C}$ be a presentable stable
$\infty$-category equipped with a $t$-structure that is compatible with filtered
colimits. 
We say that $\mathcal{C}$ is \emph{regular coherent} if $\mathcal{C}$
is compactly generated and the $t$-structure
on $\mathcal{C}$ restricts to a \emph{bounded} $t$-structure on the compact
objects $\mathcal{C}^{\omega} \subset \mathcal{C}$. 
In particular, this implies that all compact objects are truncated, and that the
truncations of a compact object remain compact. 
Consequently, the compact objects of the
heart $\mathcal{C}^{\heartsuit}$ form an abelian category
$\mathcal{C}^{\heartsuit \omega}$, and $\mathcal{C}^{\heartsuit}$ is the
$\mathrm{Ind}$-completion of $\mathcal{C}^{\heartsuit \omega} $. 
For example, if $R$ is a weakly regular coherent ring, then the derived
$\infty$-category $\mathcal{D}(R)$ of $R$-module spectra with its usual
$t$-structure is regular
coherent. 

For each connective $\mathbb{E}_\infty$-ring $S$, one has the presentable stable $\infty$-category $\mathcal{C} \otimes S$ of
$S$-module objects in $\mathcal{C}$, which also inherits a $t$-structure where
connectivity and coconnectivity are checked along restriction of scalars
$\mathcal{C} \otimes S \to \mathcal{C}$. 
Let $\mathbb{S}[t]$ denote the suspension spectrum of the commutative monoid
$\mathbb{Z}_{\geq 0}$ and $\mathbb{S}[t_1, \dots, t_n] $ be the $n$-fold smash
product of $\mathbb{S}[t]$; similarly we define $\mathbb{S}[t_1^{\pm 1}, \dots,
t_n^{\pm 1}]$ by inverting the generators. 
We say that $\mathcal{C}$
is \emph{stably regular coherent} if $\mathcal{C} \otimes
\mathbb{S}[t_1, \dots, t_n]$ is regular coherent for each $n \geq 0$. 
For example, if $R$ is a weakly regular stably coherent ring, then
$\mathcal{D}(R)$  is stably regular coherent. 
An important example is that there is an analog of Hilbert's basis theorem: if $\mathcal{C}$ is regular coherent and
$\mathcal{C}^{ \heart \omega}$ is noetherian, then $\mathcal{C} $ is stably regular
coherent \cite[Cor.~3.17]{AGH}. 

\begin{theorem}[Barwick \cite{Bar15} and Antieau--Gepner--Heller \cite{AGH}] 
\label{thmofheart}
Let $\mathcal{C}$ be a presentable stable
$\infty$-category equipped with a $t$-structure which is compatible with
filtered colimits. Suppose $\mathcal{C}$ is stably regular coherent. 
Then 
the natural map identifies the (connective) $\K$-theory of the abelian category
$(\mathcal{C}^{\heartsuit})^{\omega}$ with the $\K$-theory of the stable
$\infty$-category $\mathcal{C}^{\omega}$. In particular, $\K_{i}(
\mathcal{C}^{\omega}) =0 $ for $i < 0$. 
\end{theorem} 
\begin{proof} 
Our assumptions yield that $\mathcal{C}^{\omega}$ admits a bounded $t$-structure
with heart $\mathcal{C}^{\heart \omega}$. 
By the connective theorem of the heart \cite{Bar15}, it follows that 
$\K(\mathcal{C}^{\heart \omega}) = \K_{\geq 0}( \mathcal{C}^{\omega}) $. 
It remains to show that $\K_{-i}( \mathcal{C}^{\omega}) = 0$ for $i < 0$. This
is proved exactly as in 
the proof of 
\cite[Th.~3.6]{AGH}; we reproduce a sketch of the argument for the convenience of the
reader. 

We prove $\K_{i}( \mathcal{C}^{\omega}) = 0$ for $i < 0$ by
induction. 
First, by \cite[Th.~2.35]{AGH}, it follows that $\K_{-1}( \mathcal{C}^{\omega})
= 0$. Suppose $i < -1$. Consider the localization sequence in $\catst$ given by 
$\mathrm{Perf}( \mathbb{S}[x] \text{ on } 0 ) \to
\mathrm{Perf}(\mathbb{S}[x]) \to \mathrm{Perf}(\mathbb{S}[x^{\pm 1}])$. 
Tensoring with $\mathcal{C}^{\omega}$, we obtain another localization sequence
in $\catst$, leading to a fiber sequence
$$ \K( 
\mathrm{Perf}( \mathbb{S}[x] \text{ on } 0 ) \otimes \mathcal{C}^{\omega})
 \to \K( \mathrm{Perf}(\mathbb{S}[x]) \otimes \mathcal{C}^{\omega}) \to \K(
 \mathrm{Perf}( \mathbb{S}[x^{\pm 1}]) \otimes \mathcal{C}^{\omega}).
$$
As in the proof of \cite[Th.~3.6]{AGH}, it follows that 
$\K(\mathcal{C}^{\omega})$ is a retract of 
$\K(\mathrm{Perf}( \mathbb{S}[x] \text{ on } 0) \otimes \mathcal{C}^{\omega})$
and that this summand maps to zero in $\K( \mathrm{Perf}( \mathbb{S}[x]) \otimes
\mathcal{C}^{\omega})$. 
In particular, it follows that $\K_i( \mathcal{C}^{\omega})$ is a summand of
a quotient of $\K_{i+1}( \mathrm{Perf}( \mathbb{S}[x^{\pm 1}]) \otimes
\mathcal{C}^{\omega})$. However, 
our hypotheses imply that 
$\mathrm{Perf}( \mathbb{S}[x^{\pm 1}]) \otimes \mathcal{C}^{\omega}$ is the
subcategory of 
compact objects in 
$\mathcal{C} \otimes \mathbb{S}[x^{\pm 1}]$, which is stably regular coherent by
assumption. 
Inductively, it follows 
$\K_{i+1}( \mathrm{Perf}( \mathbb{S}[x^{\pm 1}] )\otimes \mathcal{C}^{\omega}) =
0$, whence $\K_i(\mathcal{C}^{\omega}) =0 $ as desired. 
\end{proof} 

\begin{corollary} 
\label{KisKHwhenboundedt}
Let $\mathcal{C}$ be a presentable $H\mathbb{Z}$-linear stable $\infty$-category
equipped with a $t$-structure which is compatible with filtered colimits, and
suppose $\mathcal{C}$ is stably regular coherent. Then $\K(\mathcal{C}^{\omega})
\xrightarrow{\sim} \KH( \mathcal{C}^{\omega})$. 
\end{corollary} 
\begin{proof} 
It suffices to show that $\K( \mathcal{C}^{\omega}) \xrightarrow{\sim} \K( (
\mathcal{C} \otimes_{\mathbb{Z}} \mathbb{Z}[x])^{\omega})$, and the classical proof in
algebraic $\K$-theory works. 
We switch now to geometric notation (writing schemes instead of rings). 
By our assumptions above, we have a fiber sequence
\begin{equation} \label{devA1gm} \K( \mathcal{C}^{\omega}) \to \K(\mathrm{Perf}(
\mathbb{P}^1_{\mathbb{Z}})  \otimes_{\mathbb{Z}} \mathcal{C}^{\omega} ) \to \K(
\mathrm{Perf}(\mathbb{A}^1_{\mathbb{Z}}) \otimes_{\mathbb{Z}} 
\mathcal{C}^\omega 
), \end{equation}
where the first map is obtained from the pushforward $\perf( \mathbb{Z}) \to \perf(
\mathbb{P}^1_{\mathbb{Z}})$ at the section $\infty$. 
Indeed, this follows from Thomason--Trobaugh localization \cite{TT90} and 
\Cref{thmofheart}, since 
$
\perf( \mathbb{P}^1_{\mathbb{Z}} \text{ on } \infty ) \otimes_{\mathbb{Z}}
\mathcal{C}^{\omega} $ 
is the compact objects of the stably regular coherent 
$\infty$-category of $y$-torsion objects in $\mathcal{C} \otimes_{\mathbb{Z}}
\mathbb{Z}[y]$ (for $y$ a coordinate near $\infty$ on
$\mathbb{P}^1_{\mathbb{Z}}$). 
In particular, this gives a bounded $t$-structure on 
$
\perf( \mathbb{P}^1_{\mathbb{Z}} \text{ on } \infty ) \otimes_{\mathbb{Z}}
\mathcal{C}^{\omega} $
whose heart is the category of objects in $\mathcal{C}^{\heartsuit \omega}$ with a nilpotent
endomorphism; now use Quillen d\'evissage 
\cite[Sec.~5]{Qui72}
to identify its $\K$-theory
with that of $\mathcal{C}^{\omega}$. 

Now we have two maps $f_1, f_2 \colon \mathcal{C}^{\omega} \to 
\mathrm{Perf}(\mathbb{P}^1_{\mathbb{Z}}) \otimes_{\mathbb{Z}} 
\mathcal{C}^{\omega} 
$ given by tensoring with the structure sheaf and
given by pushing forward at $\infty$. By the projective bundle formula for
$\mathbb{P}^1_{\mathbb{Z}}$, these establish an equivalence 
$(f_1, f_2) \colon \K( \mathcal{C}^{\omega})^{\oplus 2} \xrightarrow{\sim} \K(
\mathcal{C}^{\omega} \otimes_{\mathbb{Z}} \mathrm{Perf}( \mathbb{P}^1_{\mathbb{Z}}))$. 
Combining this with \eqref{devA1gm}, we find that 
pullback induces an equivalence $\K( \mathcal{C}^\omega) \xrightarrow{\sim} \K(
\mathcal{C}^\omega \otimes_{\mathbb{Z}} \mathrm{Perf}( \mathbb{A}^1_{\mathbb{Z}}))$ as
desired. 
\end{proof}

Finally, we explain how to deduce \Cref{Quillendevcoherent} and  a
generalization of 
\Cref{locregdivisor}. 

\begin{proof}[Alternative proof of \Cref{Quillendevcoherent}]
By Thomason--Trobaugh localization \cite{TT90}, the fiber of $\K(R) \to \K(
\spec R \setminus V(I))$ is given by the $\K$-theory of the stable $\infty$-category 
$\mathrm{Perf}(R \text{ on } I)$ of 
perfect $R$-module spectra which are $I$-power torsion. 
These are the compact objects in the presentable stable
$\infty$-category $\mathcal{D}(R)_{I-\mathrm{tors}} \subset \mathcal{D}(R)$ of
$I$-power torsion objects.  
Our assumption implies that 
the usual $t$-structure
on 
$\mathcal{D}(R)_{I-\mathrm{tors}}$
is stably regular coherent in the above sense. 
In particular, $\mathrm{Perf}(R \text{ on } I)$
has a bounded $t$-structure 
with heart given by the abelian category $\modfg(R)_{I-\mathrm{tors}}$ of finitely presented $R$-modules which are
$I$-power torsion. 
It follows from \Cref{thmofheart} and Quillen d\'evissage \cite[Sec.~5]{Qui72} that 
$\G(R/I) \simeq \K( \modfg(R)_{I-\mathrm{tors}}) \simeq \K
(\mathrm{Perf}(R \text{ on } I))$. 
\end{proof} 

\begin{proposition} \label{prop:torsion}
Let $R$ be a ring and let $I \subset R$ be a finitely generated regular ideal. Suppose
$R/I$ is stably coherent and weakly regular. 
Then there is a fiber sequence of spectra
\[ \K(R/I) \to \K(R) \to \K( \spec (R) \setminus V(I)),  \]
and similarly in $\KH$. 
\end{proposition} 
\begin{proof} 
We have a fiber sequence
$\K( R \text{ on } I) \to \K(R) \to \K( \spec(R) \setminus V(I)) $. 
Thus, it suffices to show that 
$\K(R/I) \xrightarrow{\sim} \K( R \text{ on } I)$. 

Consider the presentable stable $\infty$-categories
$\mathcal{D}(R)_{I-\mathrm{tors}}$ of $I$-power torsion objects in
$\mathcal{D}(R)$ and $\mathcal{D}(R/I)$. 
Then $\perf(R \text{ on } I)$ is the compact objects in 
$\mathcal{D}(R)_{I-\mathrm{tors}}$ 
and $\perf(R/I)$ is the compact objects in $\mathcal{D}(R/I)$.  
We show that $\mathcal{D}(R)_{I-\mathrm{tors}}, \mathcal{D}(R/I)$ (with the
natural $t$-structures) are stably
regular coherent. For $\mathcal{D}(R/I)$, this is part of our assumption. 
Since the hypotheses of the result apply to a polynomial ring over $R$, it
suffices to show that 
$\mathcal{D}(R)_{I-\mathrm{tors}}$ is regular coherent, i.e., that it admits a
bounded $t$-structure.

First, since $I$ is a regular ideal, $I/I^2$ is a finitely generated projective
$R/I$-module, and $\mathrm{Sym}^i_{R/I}(I/I^2) \xrightarrow{\sim} I^i/I^{i+1}$
for each $i \geq 0$. 
We argue that $R/I^i$ is a coherent ring for each $i \geq 1$ by induction. For $i = 1$ this
is the assumption. If $R/I^{i}$ is coherent, 
then consider the short exact sequence
$0 \to I^i/I^{i+1} \to R/I^{i+1} \to R/I^i \to 0$. Our assumptions give that
$I^i/I^{i+1} $ is a finitely presented (indeed, finitely generated projective)
$R/I$-module, whence a finitely presented $R/I^i$-module \cite[Lem.~3.25(i)]{BMS1}. 
Therefore, by \cite[Lem.~3.26]{BMS1}, it follows that $R/I^{i+1}$ is coherent. 

It follows 
that we have an abelian category $\mathcal{A}$ of finitely presented $R$-modules
which are annihilated by $I^n$ for some $n \gg 0$ (it is then equivalent
to assuming they are finitely presented $R/I^n$-modules,
\cite[Lem.~3.25(i)]{BMS1}). 
We claim that the usual $t$-structure on 
$\mathcal{D}(R)_{I-\mathrm{tors}}$
restricts to a bounded $t$-structure on 
$\perf(R \text{ on } I)$ with heart $\mathcal{A}$. 
We observe that any object $M \in \mathcal{A}$ is perfect as an
$R$-module; this reduces to the case where $M$ is an $R/I$-module, when it
follows because $R/I$ is regular coherent and is perfect as an $R$-module by
regularity of $I$. 
This easily implies that the homology groups of any object of $\perf(R \text{ on
} I)$ belong to $\perf(R \text{ on } I)$, whence the claim.

Thus, we have shown that 
$\mathcal{D}(R)_{I-\mathrm{tors}}, \mathcal{D}(R/I)$ are stably regular
coherent. The hearts in the compact objects $\perf(R \text{ on } I)$ and
$\perf(R/I)$ are given by $\mathcal{A} $ and the category of finitely presented
$R/I$-modules respectively. Using Quillen d\'evissage and \Cref{thmofheart}, we
conclude. 
\end{proof} 

\begin{remark}
    Much of the argument in the proof of Proposition~\ref{prop:torsion} can be
    established instead with the recent main theorem of~\cite{burklund-levy}, which
    gives a criterion for when, for a coconnective ring spectrum $A$, the natural
    map $\pi_0(A)\rightarrow A$ induces an equivalence on $\K$-theory. The
    hypotheses of their theorem holds in the setting of the proposition for the
    natural map $R/I\rightarrow R\Map_R(R/I,R/I)$, the derived endomorphism
    spectrum of the $R$-module $R/I$. But,
    $\Perf(R\Map_R(R/I,R/I))\we\Perf(R\text{ on }I)$ by derived Morita theory.
\end{remark}

\small
\bibliographystyle{amsalpha}
\bibliography{perfectoidk}

\end{document}